\documentclass[10pt, reqno]{amsart}

\usepackage{hyperref}
\hypersetup{
	colorlinks=true,
	citecolor=blue,
	linkcolor=blue,
	filecolor=magenta,      
	urlcolor=cyan,
}

\usepackage{adjustbox}
\usepackage[bb=libus]{mathalpha}
\usepackage[margin=2.2cm]{geometry}
\usepackage{setspace}
\usepackage[capitalize,noabbrev,nameinlink]{cleveref}
\usepackage{amsmath, amsthm, amssymb, graphicx, dsfont, extarrows, enumitem}
\usepackage[T1]{fontenc}
\usepackage{textcomp}
\usepackage{lmodern}
\usepackage{microtype} 
\usepackage{amscd}
\usepackage{mathrsfs}
\usepackage{tikz-cd}
\usetikzlibrary{matrix,patterns}
\usepackage[colorinlistoftodos]{todonotes}
\usepackage{color}
\usepackage{bbm}
\usepackage{verbatim}
\usepackage[mathscr]{euscript}
\usepackage{bm}

\numberwithin{equation}{section}
\theoremstyle{plain}
\newtheorem{thm}[equation]{Theorem}
\newtheorem*{thm*}{Theorem}

\newtheorem{thmx}{Theorem}

\newtheorem{prop}[equation]{Proposition}
    
\newtheorem{cor}[equation]{Corollary}       
\newtheorem{lem}[equation]{Lemma}

\theoremstyle{definition}

\newtheorem{ex}[equation]{Example}

\newtheorem{rem}[equation]{Remark}

\newtheorem{chunk}[equation]{}
\crefformat{chunk}{(#2#1#3)}

\newtheoremstyle{remboldstyle}
 {}{}{}{}{\bfseries}{.}{.5em}{{\thmname{#1 }}{\thmnumber{#2}.}{\thmnote{ #3}}}
\theoremstyle{remboldstyle}
\newtheorem{slab}[equation]{}

\newtheoremstyle{nonemboldstyle}
 {}{}{}{}{\bfseries}{.}{.5em}{{\thmnote{#3}}}
\theoremstyle{nonemboldstyle}

\newcommand{\Z}{\mathbb{Z}}

\newcommand{\mf}{\mathfrak}
\newcommand{\Hom}{\mathrm{Hom}}

\newcommand{\msf}[1]{\mathsf{#1}}

\newcommand{\mc}[1]{\mathcal{#1}}
\newcommand{\mrm}[1]{\mathrm{#1}}
\newcommand{\mbb}[1]{\mathbb{#1}}
\newcommand{\scr}[1]{\mathscr{#1}}

\newcommand{\T}{\mathsf{T}}
\newcommand{\U}{\mathsf{U}}

\newcommand{\1}{\mathds{1}}

\renewcommand{\mod}[1]{\mathrm{Mod}_{#1}}

\newcommand{\A}{\mathsf{A}}
\newcommand{\B}{\mathsf{B}}
\newcommand{\C}{\mathsf{C}}
\newcommand{\D}{\mathsf{D}}

\renewcommand{\mod}[1]{\msf{mod}(#1)}
\newcommand{\Mod}[1]{\msf{Mod}(#1)}
\newcommand{\Flat}[1]{\msf{Flat}(#1)}
\newcommand{\ab}{\msf{Ab}}

\newcommand{\ti}{\textit}

\newcommand{\rlim}{\varinjlim}
\newcommand{\llim}{\varprojlim}
\renewcommand{\t}{\text}
\renewcommand{\c}{\mrm{c}}

\newcommand{\op}{\mrm{op}}

\renewcommand{\Flat}[1]{\msf{Flat}(#1)}

\newcommand{\y}{\msf{y}}
\renewcommand{\bar}{\overline}

\renewcommand{\hat}{\widehat}

\renewcommand{\phi}{\varphi}

\newcommand{\Irr}[1]{\msf{IrrRk}(\T^\c)}

\usetikzlibrary{babel}
\title{Definable functors and Brown--Adams representability}
\author{Isaac Bird}

\address{Department of Algebra, Faculty of Mathematics and Physics, Charles University in Prague, Sokolovsk\'{a} 83, 186 75 Praha, Czech Republic}
\email{bird@karlin.mff.cuni.cz}

\begin{document}
\maketitle
\vspace{-6ex}
\begin{abstract}
The question of when the derived category of a ring satisfies Brown--Adams representability is revisited via studying the transfer of pure homological dimension along definable functors: it is shown that, for any ring, the pure global dimension of the derived category is at least the pure global dimension of the ring; expanding results of Beligiannis and Keller-Christensen-Neeman. This result is obtained by constructing `change of category' isomorphisms of PExt groups across definable functors. The same isomorphisms illustrate circumstances when one can transfer the property of Brown--Adams representability. We demonstrate how these methods can be used to test whether certain derived category of quasi-coherent sheaves are a Brown category. We also make an investigation into the structure of derived categories of von Neumann regular rings, which are shown in many cases to control Brown--Adams representability; this includes a new proof of the telescope conjecture, and a new and short proof that a coherent ring satisfies Freyd's (strong) generating hypothesis if and only if it is von Neumann regular.
\end{abstract}

\section{Introduction}
A renowned result in homotopy theory, due to Brown \cite{brown} and Adams \cite{adams}, is that the stable homotopy category of spectra $\msf{Sp}$ satisfies Brown--Adams representability. This means that any cohomological functor $F\colon(\msf{Sp}^{\omega})^{\op}\to\Mod{\Z}$ on finite spectra $\msf{Sp}^{\omega}$ is naturally isomorphic to $\Hom_{\msf{Sp}}(-,X)\vert_{\msf{Sp}^{\omega}}$ for some $X\in\msf{Sp}$, and if $\phi\colon F\to G$ is a natural transformation between cohomological functors, then $\phi\simeq\Hom_{\msf{Sp}}(-,f)$, where $f\colon X\to Y$ is a map in $\msf{Sp}$ between objects representing $F$ and $G$.

This statement can be translated into any compactly generated triangulated category $\T$: if $\T^{\c}$ denotes the compact objects of $\T$, one says that $\T$ satisfies Brown--Adams representability (or is a Brown category), provided every cohomological functor $F\colon(\T^{\c})^{\op}\to\Mod{\Z}$ is of the form $\Hom_{\T}(-,X)\vert_{\T^{\c}}$ for some object $X$, and any map $\phi\colon\Hom_{\T}(-,X)\vert_{\T^{\c}}\to\Hom_{\T}(-,Y)\vert_{\T^{\c}}$ is of the form $\Hom_{\T}(-,f)$ for some $f\colon X\to Y$. In other words, $\T$ is a Brown category, by definition, if and only if the restricted Yoneda functor
\[
\y\colon\T\to\Flat{\T^{\c}}, \, X\mapsto \Hom_{\T}(-,X)\vert_{\T^{\c}}
\]
is essentially surjective and full. Here $\Flat{\T^{\c}}$ denotes the category of cohomological functors $(\T^{\c})^{\op}\to\ab$, which is expressed this way as it is the ind-completion of the category $\T^{\c}$.

It is known that not every compactly generated $\T$ is a Brown category, and determining when a specific one is, or is not, is not elementary. Criteria to approach this problem were proved by Neeman \cite{neemantba} and later expanded by Beligiannis \cite{belrhap}: it was shown that $\T$ satisfies Brown--Adams representability if and only if $\msf{pgdim}(\T)$, the pure global dimension of $\T$, is at most one.  

It is through considering pure global dimension, or more specifically its ascent and descent, that we approach the question of whether or not a triangulated category is Brown. Our main investigation is in revisiting the question of when $\D(R)$ is a Brown category for a ring $R$, and before giving more details and stating results, we provide some historical context.

The question of $\D(R)$ being Brown was first considered in Neeman \cite{neemantba}, where an example credited to Keller showed $\D(\mbb{C}[x,y])$ is not Brown; the pertinent point is that this ring has a flat module of projective dimension at least two. In \cite{belrhap}, Beligiannis gave some further counterexamples, and provided a relationship between the pure global dimension of $\Mod{R}$ and $\D(R)$ in the case that $R$ is a (right) hereditary ring, showing the two are equal. This (among other questions) was furthered by Christensen, Keller and Neeman in \cite{ckn} to show that if $R$ is a coherent ring such that every finitely presented $R$-module has finite projective dimension, then $\msf{pgdim}(R)\leq\msf{pgdim}(\D(R))$. In this paper, the main result in this direction is the following, which removes all assumptions on rings.

\begin{thmx}\label{intro:thma}(\ref{prop:pgdimineq},\ref{prop:ringepi}, \ref{prop:hringepi})
Let $R$ be a ring, then there is an inequality $\msf{pgdim}(\D(R))\geq \msf{pgdim}(R)$. As such, $\D(R)$ is not a brown category whenever $\msf{pgdim}(R)\geq 2$. Moreover, if $f\colon R\to S$ a ring homomorphism, then
\begin{enumerate}
\item if $f$ is a ring epimorphism, one has $\msf{pgdim}(S)\leq \msf{pgdim}(R)$;
\item if $f$ is a homological ring epimorphism, one has $\msf{pgdim}(\D(S))\leq \msf{pgdim}(\D(R))$.
\end{enumerate}
\end{thmx}

As such, we see that the pure global dimension of $\D(R)$ is bounded below by that of $\Mod{R}$ for any ring, and this is an optimal lower bound. As an example of a consequence, see \ref{ex:fdalgebras}, one obtains that a finite dimensional local algebra over an algebraically closed field has a Brown derived category if and only if every module is pure injective. We also note that the first itemised statement recovers some classic results in \cite{BL}.

Let us now discuss the pure structure and pure global dimension. The restricted Yoneda embedding $\y$ described above sends triangles to long exact sequences; a triangle is sent to a short exact sequence precisely when the connecting morphism vanishes under $\y$, that is, the connecting morphism is a phantom map. These are the pure triangles, and thus the pure structure of $\T$ is the intersection of the short exact sequences in $\Mod{\T^{\c}}$, the category of additive functors $(\T^{\c})^{\op}\to\ab$, and the triangles in  in $\T$, or equivalently purity can be thought of as the study of phantom maps. 

In fact, $\y$ establishes an equivalence between the pure structure on $\T$ and the pure structure on $\Flat{\T^{\c}}$, and the latter has enough pure injective and pure projective objects. Accordingly, one can construct resolutions in $\Flat{\T^{\c}}$ with respect to these objects from which one defines the pure projective and pure injective dimensions, which in turn yield the pure global dimension of $\Flat{\T^{\c}}$. The pure global dimension  of $\T$ is defined to be the pure global dimension of $\Flat{\T^{\c}}$.

The main tool used to prove \cref{intro:thma}, among other things, is how pure homological dimension transfers across functors which preserve the pure structure. These are the definable, or interpretation, functors, which, in the triangulated setting were introduced in \cite{bwdeffun}. A functor $F\colon\T\to \U$ between compactly generated triangulated categories is definable if (and only if) it preserves coproducts, products, and pure triangles. Any such functor induces a unique product and direct limit preserving functor $\hat{F}\colon\Flat{\T^{\c}}\to\Flat{\U^{\c}}$ that is compatible with $\y$, and as such, understanding the transfer of pure homological dimension in this accessible setting is good enough to understand what happens on the triangulated level. This is our first main result, which is the workhorse of the paper.

\begin{thmx}\label{intro:thmb}(\ref{prop:pextisos})
Let $F\colon\A\to\B$ be a definable functor between finitely accessible categories with products, and suppose that $F$ admits a left adjoint $\Lambda\colon\B\to\A$. Then there are isomorphisms
\[
\mrm{PExt}_{\msf{B}}^{n}(B,F(A))\simeq \mrm{PExt}_{\A}^{n}(\Lambda(B),A)
\]
for all $A\in\A, B\in\B$ and $n\geq 0$.
\end{thmx}

Here, $\t{PExt}$ denotes the derived functors of Hom with respect to pure projective and pure injective resolutions, which measure the pure homological dimension of objects. We note that any definable functor of triangulated categories, as many other definable functors which appear naturally, translate into the above framework. In relation to \cref{intro:thma}, it is through the above isomorphisms applied to particular choices of definable functors between $\Mod{R}$, $\D(R)$, and $\D(S)$ that one obtains the result.

However, while the inequalities of \cref{intro:thma} are more suited to illustrate when Brown--Adams representability fails, we are still able to use them, as well as \cref{intro:thmb}, to provide certain positive examples of when it holds.

\begin{thmx}\label{intro:thmc}(\ref{thm:BAfunctoriality})
Let $\T$ and $\U$ be compactly generated triangulated categories, and suppose $F\colon\Flat{\T^{\c}}\to\Flat{\U^{\c}}$ is a definable functor. 
\begin{enumerate}
\item If $F$ is essentially surjective up to direct summands and $\T$ is a Brown category, so is $\U$;
\item If $F$ is fully faithful and $\U$ is a Brown category, then so is $\T$.
\end{enumerate}
\end{thmx}

The combination of \cref{intro:thma} and \cref{intro:thmc} also enables a straightforward test as to whether, for example, the category $\D_{\t{qc}}(X)$ fails to be a Brown category, where $X$ is a quasi-compact quasi-separated scheme. This is done in \cref{prop:brownschemes}, where it is shown that if $\D_{\t{qc}}(X)$ is Brown, then $\msf{pgdim}(\scr{O}_{X,x})\leq 1$ for all $x\in X$. 

Returning to derived categories of rings, as \cref{intro:thma} illustrates, understanding ring epimorphisms is enough to determine whether Brown representability fails. And in fact, for many rings, such as all commutative rings, it is possible to test this on a single ring epimorphism. A classic result of Olivier \cite{olivier} demonstrates how every commutative ring $A$ admits a unique non-zero epimorphism $A\to A^{\t{ab}}$ where $A^{\t{ab}}$ is a commutative von Neumann regular ring. These rings are particularly relevant since they are those for which pure global dimension coincides with the usual global dimension. Olivier's result was recently expanded to all rings in \cite{herzinn}, and thus for any ring $R$ there is a ring epimorphism $R\to R^{\t{ab}}$ where $R^{\t{ab}}$ is an abelian von Neumann regular ring (although it can be the case in this generality that $R^{\t{ab}}$ vanishes). As such, as shown in \cref{prop:vnrdeterminesbrown}, von Neumann regular rings determine whether derived categories of many rings are Brown categories.

Along the way to getting to this point, we investigate some parts of the structure of the derived category of von Neumann regular rings. As a consequence, we obtain, among other things, a new and very short proof that a coherent ring satisfies Freyd's generating hypothesis if and only if it is von Neumann regular (see  \cref{prop:fghholds}), as well as a new proof of the telescope conjecture for such rings in \cref{prop:telescope}. Lastly, in \cref{sec:beyondderived}, we show how some of the results concerning bounding the pure global dimension derived categories of rings can be extended to, in some cases, all compactly generated triangulated categories; this includes a consideration of the hearts of cosilting t-structures in \cref{prop:cosilting}.

\subsection*{Acknowledgements}
This research was supported by project PRIMUS/23/SCI/006 from Charles University, and by Charles University Research Centre program No. UNCE/24/SCI/022. I am grateful to M. Prest, S. Virili, J. Williamson and A. Zvonareva for constructive comments.

\section{Preliminaries on modules}
Throughout, all categories will be assumed to be preadditive, and all functors will be assumed to be additive. For brevity, we let $\ab$ denote the category of abelian groups. For categories $\A$ and $\B$, we let $(\A,\B)$ denote the functor category of all additive functors $\A\to \B$.

\begin{slab}[Module categories]\label{prel:modules} Let $\scr{R}$ be a small preadditive category. 
By definition, a right $\scr{R}$-module is an additive functor $\scr{R}^{\op}\to\ab$. The category of right $R$-modules will be denoted $\Mod{\scr{R}}$. This is a locally finitely presented Grothendieck abelian category. The full subcategory of finitely presented $\scr{R}$-module will be denoted $\mod{\scr{R}}$. A module $M\in\Mod{\scr{R}}$ is finitely presented if and only if it has a presentation 
\[
\bigoplus_{i=1}^{n}\Hom_{\scr{R}}(-,A_{i})\longrightarrow  \bigoplus_{j=1}^{m}\Hom_{\scr{R}}(-,B_{j})\longrightarrow  M\to 0
\]
where $A_{i},B_{j}\in\scr{R}$. A module $M\in\Mod{\scr{R}}$ is \emph{flat} if it is isomorphic to a direct limit of representable functors, that is $M\simeq\rlim_{i}\Hom_{\scr{R}}(-,R_{i})$, where each $R_{i}\in\scr{R}$. An object $X\in\Mod{\scr{R}}$ is \emph{fp-injective} provided $\t{Ext}^{1}(M,X)=0$ for all $M\in\mod{\scr{R}}$. We shall write $\msf{Flat}(\scr{R})$ and $\msf{fpInj}(\scr{R})$ for the full subcategories of flat and fp-injective modules, respectively. The full subcategory of injective objects will likewise be denoted $\msf{Inj}(\scr{R})$.

Given an additive functor $f\colon\scr{R}\to\scr{S}$ between small preadditive categories, there is an induced adjoint triple
\[
\begin{tikzcd}[column sep = 1.5cm]
\Mod{\scr{S}} \arrow[r, "f^{*}" description] \arrow[r, leftarrow, shift left = 1.5ex, "f^{!}"] \arrow[r, leftarrow, shift right = 1.5ex, swap, "f_{*}"]& \Mod{\scr{R}}
\end{tikzcd}
\]
where $f^{*}(M)=M\circ f$ is the restriction of scalars, and writing $F\in\Mod{R}$ as $M\simeq\msf{colim}_{I}\Hom_{\scr{R}}(-,R_{i})$ via coproducts and cokernels, one has
\[
f^{!}(M)=\msf{colim}_{I}\Hom_{\scr{S}}(-,f(R_{i})).
\]
\end{slab}

\begin{slab}[Purity]
Let $\A$ be a finitely accessible category with products, and set $\C=\msf{fp}(\A)$. This means that $\A$ is a category with all direct limits, and for each $A\in\A$ there is an isomorphism $A\simeq\rlim_{I}C_{i}$ with $C_{i}\in\C$.

A short exact sequence $0\to X\to Y\to Z\to 0$ in $\A$ is \emph{pure} if $0\to (A,X)\to\Hom_{\A}(A,Y)\to\Hom_{\A}(A,Z)\to 0$ is exact for all $A\in\msf{fp}(\A)$. In this case we say that the map $X\to Y$ is a \emph{pure monomorphism} and that $X$ is a \emph{pure subobject} of $Y$, and similarly define the notions of pure subobjects and quotients.

An object $X\in\A$ is \emph{pure injective} if any pure monomorphism $X\to Y$ splits, while $P\in\A$ is \emph{pure projective} if any pure epimorphism $M\to P$ splits. The full subcategory of pure injective objects in $\A$ will be denoted $\msf{Pinj}(\A)$, and likewise define $\msf{Pproj}(\A)$.
\end{slab}

\begin{ex}
Any epimorphism $M\to F$, with $F\in\Flat{\scr{R}}$ is pure, while any monomorphism $X\to Y$, with $X\in\msf{fpInj}(\scr{R})$ is pure, see \cite[Propositions 5.6 and 5.10]{DAC}.
\end{ex}

\begin{slab}[Triangulated categories]\label{prel:puretri}
If $\T$ is a compactly generated triangulated category, then the subcategory of compact objects will be denoted $\T^{\c}$. There is an equivalence of categories $\Flat{\T^{\c}}=\msf{fpInj}(\T^{\c})$, and these coincide precisely with the category of cohomological functors, see \cite[Lemma 2.7]{krsmash}. As such, the \emph{restricted Yoneda embedding} 
\[
\y\colon X\mapsto \Hom_{\T}(-,X)\vert_{\T^{\c}}
\]
can be viewed as a functor $\y\colon\T\to\Flat{\T^{\c}}$. 

Purity in $\T$ is defined in terms of purity in $\Flat{\T^{\c}}$. A triangle $X\to Y\to Z\to \Sigma X$ is \emph{pure} if and only if $0\to \y X\to \y Y\to \y Z\to 0$ is pure in $\Flat{\T^{\c}}$. Similarly, an object $X\in\T$ is \emph{pure injective} if and only if $\y X$ is pure injective in $\Flat{\T^{\c}}$, which is equivalent to being injective in $\Mod{\T^{\c}}$. We let $\msf{Pinj}(\T)$ denote the class of pure injective objects in $\T$, and note there is an equivalence of categories $\y\colon\msf{Pinj}(\T)\xrightarrow{\simeq}\msf{Pinj}(\Flat{\T^{\c}})=\msf{Inj}(\T^{\c})$.

We note that $\y$ is seldom fully faithful, as its kernel consists of the phantom maps. However, if $T\in\T$ and $X\in\msf{Pinj}(\T)$, the canonical map $\Hom_{\T}(T,X)\to \Hom(\y T,\y X)$ is a bijection, and this is also true whenever $T\in\msf{Pproj}(\T)$ and $X\in\T$ is arbitrary, see \cite[Theorem 1.8]{krsmash}.
\end{slab}

\begin{slab}[Definable functors]\label{prel:deffunctors}
Central to our study will be functors which transfer purity. The appropriate notion for this is that of a definable functor. A functor $F\colon \A\to \B$ between finitely accessible categories with products is \emph{definable} if it preserves filtered colimits and products. This is the appropriate notion of a purity preserving functor, since, as illustrated in \cite[\S 13]{DAC}, definable functors preserve pure exact sequences and pure injective objects. The connection between definable functors and model-theoretic methods is described at the same place.

The corresponding construction for compactly generated triangulated categories was introduced in \cite{bwdeffun}. A functor $F\colon\T\to\U$ between compactly generated triangulated categories is definable if there is a (unique) definable functor $\widehat{F}\colon\Flat{\T^{\c}}\to\Flat{\U^{\c}}$ such that $\widehat{F}\circ\y\simeq\y\circ F$; this is equivalent to $F$ preserving products, coproducts, and pure triangles, see \cite[Theorem 4.16]{bwdeffun}.
\end{slab}

\begin{slab}[From finitely accessible to locally coherent categories]\label{prel:catDA}
A \emph{locally coherent} abelian category is a finitely accessible abelian category whose finitely presented objects also form an abelian category. Given $\A$, a finitely accessible category with products, Crawley-Boevey introduced a unique embedding of $\A$ into a locally coherent category $\mbb{D}(\A)$, see \cite[\S 3.3]{CBlfp}. We recall its construction here.

Let $\C=\msf{fp}(\A)$, and identify $\A\simeq\msf{Flat}(\C)$, which we may do by \cite[Theorem 1.4]{CBlfp}. The Yoneda embedding $\msf{h}\colon\C \to \msf{fp}(\C,\ab)^{\op}$ given by $C\mapsto\Hom_{\A}(C,-)$ induces, by \cref{prel:modules}, a functor 
\[
\msf{h}^{!}\colon \Mod{\C}\to\Mod{\msf{fp}(\C,\ab)^{\op}}\simeq (\msf{fp}(\C,\ab),\ab)
\] 
which, being the left adjoint to the exact functor $\msf{h}^{*}$, preserves colimits and projective objects, and therefore also flat objects. As such, there is a functor
\[
d_{\A}\colon \A\to\mbb{D}(\A):=\Flat{\msf{fp}(\C,\ab),\ab}
\]
which, as illustrated in \cite[\S3.4 Lemma 2]{CBlfp}, is fully faithful and definable. The category $\mbb{D}(\A)$ is locally coherent, and the functor $d_{\A}$ restricts to give equivalences
\[
d_{\A}\colon\A\xrightarrow{\simeq}\msf{fpInj}(\mbb{D}(\A))
\]
and $\msf{Pinj}(\A)\simeq\msf{Inj}(\mbb{D}(\A))$. We note that a sequence $0\to X\to Y\to Z\to 0$ in $\A$ is pure if and only if its image under $d_{\A}$ is exact in $\mbb{D}(\A)$.

As illustrated in \cite[Corollary 9.8]{krausefunlfp}, the assignment $\A\mapsto \mbb{D}(\A)$ gives a bijection, up to equivalence, between finitely accessible categories with products and locally coherent abelian categories whose fp-injective objects are finitely accessible; hence we may uniquely pass from $\A$ to $\mbb{D}(\A)$ and back.
\end{slab}

\begin{ex}
The category $\Mod{\scr{R}}$ is locally coherent if and only if $\scr{R}$ has pseudokernels. This is certainly the case when $\scr{R}$ is a triangulated category, where pseudokernels are just cocones. If $R$ is a ring, $\Mod{R}$ is locally coherent if and only if $R$ is right coherent. See \cite[\S 6]{DAC} for more on locally coherent categories.
\end{ex}

\begin{slab}[Pure homological dimensions]
As usual, let $\A$ denote a finitely accessible category with products. As it is locally coherent, the category $\mbb{D}(\A)$ has enough injectives, and since $d_{\A}$ is fully faithful it follows that $\A$ has pure injective envelopes, that is there is a pure monomorphism $X\to PE(X)$ for each $X\in\A$. More specifically, taking the injective envelope $0\to d_{A}X\to d_{A}E\to \Omega^{-1}X\to 0$ in $\mbb{D}(\A)$, and noting that $\Omega^{-1}X\in\msf{fpInj}(\mbb{D}(\A))$, the fully faithfulness of $d_{A}$ enables us to pull the exact sequence back to a pure exact sequence in $\A$.

We may therefore consider pure injective resolutions of objects of $\A$. In the same way, since $d_{\A}$ sends pure projective objects of $\A$ to projective objects in $\mbb{D}(\A)$, we may consider pure projective resolutions.

The PExt functors are defined to be the derived functors of $\Hom$ with respect to pure resolutions, see \cite[p. 95]{simsonpgd}. Similarly to usual projective and injective dimension, one can define the pure projective and pure injective dimension in terms of vanishing of $\t{PExt}_{\A}^{n}(-,-)$; explicitly
\begin{align*}
\msf{ppd}(A)=\sup\{n:\t{PExt}_{\A}^{n}(A,-)\neq 0\} & & \t{ and } & & \msf{pid}(\A)=\sup\{n:\t{PExt}_{\A}^{n}(-,A)\neq 0\}.   
\end{align*}
The \emph{pure global dimension} of $\A$ is defined to be
\[
\msf{pgdim}(\A)=\sup\{\msf{ppd}(A):A\in\A\}=\sup\{\msf{pid}(A):A\in\A\}.
\]
\end{slab}

\begin{slab}[Inverse limits of injective objects]\label{prel:bergman}
We saw above that the direct limit closure of the projective objects are the flat objects. We shall also frequently use the dual result, which is a generalisation of a result of Bergman. For a ring $R$, Bergman's result \cite[Theorem 2]{bergman} states that every $R$-module is isomorphic to an inverse limit of injective modules. In fact, it was shown in \cite[Proposition 2.3]{ban} that the same result holds in any Grothendieck category.
\end{slab}

\begin{slab}[Cotorsion pairs]
Let $\scr{A}$ be a locally coherent Grothendieck category and $A\in\scr{A}$. For any $i\in \Z$ we set $A^{\perp_{i}}=\{M\in\scr{A}:\t{Ext}_{\scr{A}}^{i}(A,M)=0\}$, and define various other decorations similarly and extend these to classes of objects. A pair of classes $(\scr{C},\scr{D})$ in $\scr{A}$ is a \emph{cotorsion pair} if $\scr{C}^{\perp_{1}}=\scr{D}$ and $\scr{C}=\,^{\perp_{1}}\scr{D}$. A cotorsion pair in $\scr{A}$ is said to be \emph{hereditary} if $\msf{Inj}(\scr{A})\subseteq\scr{D}$ and $\scr{D}$ is closed under cokernels of monomorphisms. This implies that $\t{Ext}_{\scr{A}}^{i}(\scr{C},\scr{D})=0$ for all $i>0$, see \cite[Theorem 2.16]{gillespiebook}. A cotorsion pair in $\scr{A}$ is \emph{complete} if for all $A\in \scr{A}$ there are short exact sequence $0\to D\to C\to A\to 0$ and $0\to A\to C'\to D'\to 0$ with $C,C'\in\scr{C}$ and $D,D'\in\scr{D}$.
\end{slab}

\begin{ex}\label{ex:flatcotpair}
Let $\scr{R}$ be a skeletally small preadditive category. An object $G\in\Mod{\scr{R}}$ is \emph{cotorsion} if $\t{Ext}_{\scr{R}}^{1}(F,G)=0$ for all $F\in\Flat{\scr{R}}$, and we let $\msf{Cot}(\scr{R})$ denote the class of cotorsion objects in $\Mod{\scr{R}}$. It is shown in \cite[Proposition 4.7]{stovicekpure} that $(\Flat{\scr{R}},\msf{Cot}(\scr{R}))$ is a complete hereditary cotorsion pair in $\Mod{\scr{R}}$.
\end{ex}

We now briefly recall the localisation theory of abelian categories. More detailed expositions can be found in, for instance, \cite{psl}, \cite[\S7]{DAC} or \cite{stenstrom}. In particular, proofs of all the recollections below can be found at these references.

\begin{slab}[Localisations of locally coherent categories]\label{prel:localisation}
Let $\scr{A}$ be a locally coherent Grothendieck category. A full subcategory $\scr{S}\subseteq\scr{A}$ is \emph{Serre} if given a short exact sequence $0\to A_{1}\to A_{2}\to A_{3}\to 0$ in $\scr{A}$, one has $A_{1},A_{3}\in\scr{A}$ if and only if $A_{2}\in\scr{A}$. Associated to a Serre subcategory is an exact localisation functor $Q_{\scr{S}}\colon\scr{A}\to\scr{A}/\scr{S}$, where $\scr{A}/\scr{S}$ is an Abelian category such that any exact functor $F\colon\scr{A}\to\A$ into an abelian category $\A$ such that $F(\scr{S})=0$ factors through $Q_{\scr{S}}$.

A subcategory $\scr{S}\subseteq\scr{A}$ is \emph{localising} if it is Serre and closed under coproducts. In this case, the localisation $Q_{\scr{S}}\colon\scr{A}\to\scr{A}/\scr{S}$ admits a right adjoint $R_{\scr{S}}\colon\scr{A}/\scr{S}\to\scr{A}$. A localising subcategory $\scr{S}$ is of \emph{finite type} if $\scr{S}=\rlim(\scr{S}\cap\msf{fp}(\scr{A}))$; this is equivalent to $Q_{\scr{S}}$ preserving finitely presented objects, or equivalently $R_{\scr{S}}$ preserving direct limits (and thus being definable). In this case $\scr{A}/\scr{S}$ is itself locally coherent.
\end{slab}

Finite type localisations are of particular interest since they are closely related to the purity in the category. We shall only use this in the context of triangulated categories, so this is recalled below.

\begin{slab}[Definable subcategories]\label{prel:definable}
Let $R$ be a ring and $\D(R)$ its derived category. A full subcategory $\mc{D}\subseteq\D(R)$ is \emph{definable} if it is closed under products, pure subobjects and pure quotients. There is a bijection between definable subcategories of $\D(R)$ and finite type localisations of $\Mod{\D(R)^{\c}}$, or equivalently Serre subcategories of $\mod{\D(R)^{\c}}$. The bijections are given by
\[
\mc{D}\mapsto \mc{S}(\mc{D})=\{f\in\mod{\D(R)^{\c}}:\Hom(f,\y X)=0 \t{ for all }X\in\mc{D}\}
\]
and 
\[
\mc{S}\mapsto \mc{D}(\mc{S})=\{X\in\D(R):\Hom(f,\y X)=0 \t{ for all }f\in\mc{S}\}.
\]
For more details concerning this, see \cite{BWhs}.
\end{slab}

\section{Transfer of pure homological dimension}

For this section, there is a standing assumption that $F\colon\A\to\B$ is a definable functor between finitely accessible categories with products. Since any definable functor preserves pure exact sequences and pure injective objects, it is clear that $\msf{pid}(FX)\leq \msf{pid}(X)$ for any $X\in\A$. However, as the following result, which is arguably the technical core of the paper, shows, to get bounds on the pure global dimension it is more convenient to consider an additional assumption on $F$, namely that it has a left adjoint.

\begin{prop}\label{prop:pextisos}
Let $F\colon\A\to\B$ be a definable functor between finitely accessible categories with products, and suppose that $F$ admits a left adjoint $\Lambda\colon\B\to\A$. Then there are isomorphisms
\[
\mrm{PExt}_{\msf{B}}^{n}(B,F(A))\simeq \mrm{PExt}_{\A}^{n}(\Lambda(B),A)
\]
for all $A\in\A, B\in\B$ and $n\geq 0$.
\end{prop}

\begin{proof}
We make substantial use of the functor $d_{\A}\colon\A\to\mbb{D}(\A)$ introduced in \cref{prel:catDA}. Since $d_{\B}$ is definable, so is the composition $d_{\B}\circ F\colon \A\to\mbb{D}(\B)$. By \cite[UP 10.4]{krausefunlfp}, there is a functor $\bar{F}\colon\mbb{D}(\A)\to\mbb{D}(\B)$, which is unique up to natural isomorphism, that extends $F$ along $d_{\A}$, so $\bar{F}\circ d_{\A}\simeq d_{\B}\circ F$; furthermore, $\bar{F}$ is definable as $F$ is, and admits an exact left adjoint $\lambda\colon\mbb{D}(\B)\to\mbb{D}(\A)$. The isomorphism $\bar{F}\circ d_{\A}\simeq d_{\B}\circ F$ shows that $\bar{F}$ preserves injective objects.

Let us fix an object $Y\in\mbb{D}(\B)$ and consider the composition
\[
\Hom_{\mbb{D}(\B)}(Y,-)\circ F\colon\mbb{D}(\A)\to\ab
\]
which is, by the adjunction $\lambda\dashv\bar{F}$, naturally isomorphic to $\Hom_{\mbb{D}(\A)}(\lambda(Y),-)$. There is then, by for example \cite[Theorem 5.8.3]{weibel}, a Grothendieck spectral sequence
\[
\t{Ext}_{\mbb{D}(\B)}^{p}(Y,(\msf{R}^{q}\bar{F})(X))\Rightarrow\t{Ext}_{\mbb{D}(\A)}^{p+q}(\lambda(Y), X)
\]
for any $X\in\mbb{D}(\A)$.

Let us now suppose $A\in\A$ and we set $X=d_{\A}A$, which is an fp-injective object in $\mbb{D}(\A)$. In the category $\msf{fpInj}(\mbb{D}(\A))$ there is no distinction between an exact sequence and a pure exact sequence, and as the category is closed under pure quotients, it follows that, since $\bar{F}$ preserves pure exact sequences, the restriction $\bar{F}\vert\colon\msf{fpInj}(\mbb{D}(\A))\to\msf{fpInj}(\mbb{D}(\B))$ is exact. Consequently $\msf{R}^{p}\bar{F}(X)=0$ for all $p>0$ and the above spectral sequence collapses to give isomorphisms 
\[
\t{Ext}_{\mbb{D}(\B)}^{n}(Y,d_{\B}F(A))\simeq\t{Ext}_{\mbb{D}(\A)}^{n}(\lambda(Y), d_{\A}A)
\]
for all $A\in\A$. 

To conclude the proof, we must show that $\lambda\circ d_{\B}\simeq d_{\A}\circ\Lambda$. To this end, let $X\in\mbb{D}(\A)$ and using \cref{prel:bergman} write $X\simeq \llim_{J}d_{\A}X_{j}$ where $X_{j}\in\msf{Pinj}(\A)$. There are then isomorphisms
\begin{align*}
\Hom_{\mbb{D}(\A)}(d_{\A}\Lambda B, X) & \simeq \llim_{J}\Hom_{\mbb{D}(\A)}(d_{\A}\Lambda B, d_{\A}X_{j}) \\
&\simeq \llim_{J}\Hom_{\mbb{D}(\B)}(d_{\B}B,d_{\B}FX_{j}) & \mbox{by $d_{\A},d_{\B}$ being fully faithful and $\Lambda\dashv F$,} \\
&\simeq \Hom_{\mbb{D}(\B)}(d_{\B}B, \llim_{J} \bar{F}d_{\A}X_{j}) & \mbox{as $d_{\B}F\simeq \bar{F}d_{A}$,}\\
&\simeq \Hom_{\mbb{D}(\A)}(\lambda d_{\B}B,X) &\mbox{since $\bar{F}$ preserves limits and $\lambda\dashv \bar{F}$}.
\end{align*}
As this holds for all $X$, there is then a natural isomorphism $\lambda\circ d_{\B}\simeq d_{\A}\circ\Lambda$. The fact that, by definition, $\t{PExt}_{\A}^{i}(A,A')\simeq\t{Ext}_{\mbb{D}(\A)}^{i}(d_{\A}A,d_{\A}A')$ for all $A,A'\in \A$, the isomorphisms in the statement hold as claimed.
\end{proof}

\begin{rem}
Note that the existence of $\Lambda$ is only required in the above to give a more closed form of the statement. The adjoint $\lambda$ always exists, and thus one will always have the isomorphisms $\t{PExt}_{\B}^{n}(B,F(A))\simeq\t{Ext}_{\mbb{D}(\A)}^{n}(\lambda(d_{B}B), d_{\A}A)$, it is just not necessarily the case that $\lambda(d_{\B}B)$ is an fp-injective object, in other words one in the image of $d_{\A}$. The existence of $\Lambda$ ensures this occurs.
\end{rem}

Let us give some immediate corollaries.

\begin{cor}\label{cor:esssurj}
With the setup of \cref{prop:pextisos}, the following hold:
\begin{enumerate}
\item $\Lambda$ preserves pure projective objects;
\item if every object in $\B$ is a summand of an object in the image of $F$, then $\msf{pgdim}(\B)\leq \msf{pgdim}(\A)$.
\end{enumerate}
\end{cor}
\begin{proof}
For the first item, note that $B\in\B$ is pure projective if and only if $\t{PExt}_{\B}^{1}(B,-)=0$. The adjunction of \cref{prop:pextisos} with $n=1$ shows $\Lambda(B)$ is pure projective whenever $B$ is. For the second item, it is enough to test on object in the image of $F$. In this case, one has $\t{PExt}_{\B}^{n}(B,X)=\t{PExt}_{\B}^{n}(B, FA)\simeq \t{PExt}_{\A}^{n}(\Lambda(B),A)$, and thus if the latter is zero, so is the former.
\end{proof}

The existence of the left adjoint $\Lambda$ is not overly restrictive. Let us give some examples illustrating this.

\begin{ex}\label{ex:tridef}
Let $F\colon \T\to \U$ be a definable functor of compactly generated triangulated categories. Then, as well as the induced definable functor $\hat{F}\colon\Flat{\T^{\c}}\to\Flat{\U^{\c}}$ described in \cref{prel:deffunctors}, there is a (unique) definable functor $\bar{F}\colon\Mod{\T^{\c}}\to\Mod{\U^{\c}}$ such that $\bar{F}\vert_{\Flat{\T^{\c}}}=\hat{F}$. Furthermore, $\bar{F}$ admits a left adjoint $\Lambda\colon\Mod{\U^{\c}}\to\Mod{\T^{\c}}$. This is illustrated in \cite[\S 4.C]{bwdeffun}.

In relation to the above remark, note that $\Lambda$ will generally not restrict to flat objects even though $\bar{F}$ does, hence one needs to consider the adjunction as definable functors between locally coherent, rather than finitely accessible, categories.
\end{ex}

\begin{ex}\label{ex:restadj}
For a functor $f\colon \scr{R}\to\scr{S}$ of small preadditive categories, the adjoint pair $f^{!}\dashv f^{*}$ defined in \cref{prel:modules} satisfies the conditions of \cref{prop:pextisos}. 
\end{ex}

\begin{ex}\label{ex:ringmap}
As a particular example of \cref{ex:restadj}, a ring map $f\colon R\to S$ induces the adjoint pair $S\otimes_{R}-\dashv (-)_{R}$, where $(-)_{R}\colon\Mod{S}\to\Mod{R}$ is the restriction of scalars, which is definable as it also has a right adjoint. The isomorphisms of \cref{prop:pextisos} then give
\begin{equation}\label{eqn:ringmap}
\t{PExt}_{R}^{i}(M,(N)_{R})\simeq \t{PExt}_{S}^{i}(M\otimes_{R}S,N)
\end{equation}
for any $M\in\Mod{R}$ and $N\in\Mod{S}$.
\end{ex}

Recall that a ring map $f\colon R\to S$ is a \ti{ring epimorphism} if the restriction $(-)_{R}\colon\Mod{S}\to\Mod{R}$ is fully faithful (see \cite[\S 5.5.1]{psl} for more details). For such ring maps, it is straightforward obtain a relationship between their pure global dimension.

\begin{prop}\label{prop:ringepi}
Let $f\colon R\to S$ be a ring epimorphism. Then $\msf{pgdim}(S)\leq \msf{pgdim}(R)$.
\end{prop}
\begin{proof}
The fact that $f$ is a ring epimorphism is equivalent to there being an isomorphism $N\simeq S\otimes_{R}(N)_{R}$ for any $N\in\Mod{S}$, see \cite[Proposition XI.1.2]{stenstrom}. As such, the isomorphisms of \cref{eqn:ringmap} give $\t{PExt}_{R}^{i}((M)_{R},(N)_{R})\simeq \t{PExt}_{S}^{i}(M,N)$ for all $M,N\in\Mod{S}$, from which it is immediate that $\msf{pgdim}(S)\leq \msf{pgdim}(R)$.
\end{proof}

We note that \cref{prop:ringepi} recovers \cite[Proposition 2.2 and Corollary 2.3]{BL}. We shall consider the analogous version of the above for homological ring epimorphisms when we consider applications of \cref{prop:pextisos} for derived categories.

In fact, the above is an example of when our standard adjunction $\Lambda\dashv F$, with $F$ definable, has one of the functors being fully faithful. In this generality, we have the following.

\begin{lem}\label{lem:localisation}
Let $F\colon \A\to \B$ be a definable functor with left adjoint $\Lambda$. 
\begin{enumerate}
    \item If $F$ is fully faithful, then $\msf{pgdim}(\A)\leq \msf{pgdim}(\B)$.
    \item If $\Lambda$ is fully faithful, then $\msf{pgdim}(\B)\leq \msf{pgdim}(\A)$.
\end{enumerate}
\end{lem}
\begin{proof}
By \cref{prop:pextisos}, there are isomorphisms
\[
\t{PExt}_{\B}^{n}(FM,FM')\simeq \t{PExt}_{\A}^{n}(\Lambda FM,M')
\]
and
\[
\t{PExt}_{\A}^{n}(\Lambda N,\Lambda N')\simeq \t{PExt}_{\B}^{n}(N,F\Lambda N')
\]
for all $M,M'\in\A$ and $N,N'\in \B$. Now, $F$ (respectively $\Lambda$) is fully faithful if and only if $\Lambda F\simeq \t{Id}_{\A}$ (respectively $F\Lambda\simeq \t{Id}_{\B}$), see \cite[Theorem IV.3.1]{maclane}. The claims are now immediate.
\end{proof}

\begin{ex}
In the mode of \cref{prel:localisation}, if $\scr{A}$ is a locally coherent category and $\mc{S}\subseteq\msf{fp}(\scr{A})$ is a Serre subcategory, then the induced localisation $Q_{\mc{S}}\colon \scr{A}\to\scr{A}/\rlim\mc{S}$ is of finite type, so its right adjoint preserves direct limits and is therefore definable. Consequently, we have that $\msf{pgdim}(\scr{A}/\rlim\mc{S})\leq \msf{pgdim}(\scr{A})$.

On the other hand, if $f\colon \scr{R}\to\scr{S}$ is a fully faithful functor between finitely accessible categories, then the adjunction $f^{!}\dashv f^{*}$ of \cref{prel:modules} has $f^{!}$ being fully faithful, and thus $\msf{pgdim}(\Mod{\scr{R}})\leq \msf{pgdim}(\Mod{\scr{S}})$.
\end{ex}

\begin{ex}
In certain situations where the setup of \cref{prop:pextisos} holds, the left adjoint $\Lambda\colon\B\to\A$ is itself a definable functor. For example, this occurs for \cref{ex:ringmap} when $f\colon R\to S$ realises $S$ as a finitely presented $R$-module. In this case, the adjunction $\Lambda\vdash F$ extends along $d_{\A}$ and $d_{\B}$ to give an adjunction $\bar{\Lambda}\dashv \bar{F}$. Indeed, if $\lambda\dashv \bar{F}$ is the exact left adjoint as in the proof of \cref{prop:pextisos}, then $\lambda\circ d_{\B}\simeq d_{\A}\circ \Lambda$ - this can be tested by using the uniqueness of left adjoints. Since $\lambda$ is exact, and thus left exact, and $\bar{\Lambda}$ is left exact, applying both functors the start of an injective resolution of any $X\in\mbb{D}(B)$, say $0\to X\to d_{\B}E_{1}\to d_{\B}E_{2}$, shows that $\lambda = \bar{\Lambda}$. In particular, $\bar{\Lambda}$ is an exact functor with adjoints on both sides.
\end{ex}

\begin{ex}
Let $\T_{1},\T_{2}$ and $\T_{3}$ be compactly generated triangulated categories, suppose that there is a sequence of definable functors $\Flat{\T_{1}^{\c}}\hookrightarrow\Flat{\T_{2}^{\c}}\twoheadrightarrow\Flat{\T_{3}^{\c}}$; this happens, for instance, when $\T_{1}$ is a smashing subcategory of $\T_{2}$ and $\T_{3}$ is the corresponding localisation, as a particular form of \cref{ex:tridef}. By \cref{lem:localisation}, we then have $\msf{pgdim}(\T_{1}),\msf{pgdim}(\T_{3})\leq \msf{pgdim}(\T_{2})$. Yet in general there is no way to provide an upper bound of $\msf{pgdim}(\T_{2})$ in terms of $\msf{pgdim}(\T_{1})$ and $\msf{pgdim}(\T_{3})$. For example, consider $\D(A)$ where $A$ is the path algebra of the Kronecker quiver over a field $k$; there is then a recollement $\D(k)\hookrightarrow\D(A)\twoheadrightarrow\D(k)$, see \cite[Remark 2.11]{AKLY}. As $A$ is hereditary, we have $\msf{pgdim}(\D(A))\in\{1,2\}$ by \cite[Corollay 11.33]{JL}, depending on whether $k$ is countable or uncountable, while $\msf{pgdim}(\D(k))=0$ in any circumstance.
\end{ex}

\subsection{Brown--Adams representability}
We now turn our attention to an application of the transfer of pure global dimension to compactly generated triangulated categories. From now on, this assumption of compact generation will be standing.

\begin{slab}[Brown categories]
A compactly generated triangulated category $\T$ is said to be \emph{Brown}, or satisfy \emph{Brown--Adams representability} provided that the restricted Yoneda embedding $\y\colon\T\to\Flat{\T^{\c}}$ is full and essentially surjective. 

In other words, $\T$ is Brown if and only if any cohomological functor $H\colon(\T^{\c})^{\op}\to\ab$ is of the form $\y X=\Hom_{\T}(-,X)\vert_{\T^{\c}}$ for some $X\in\T$, and if $\phi\colon\y X\to \y Y$ is a natural transformation, then there is a map $f\colon X\to Y$ inducing $\phi$.
\end{slab}

\begin{rem}\label{rem:exofbrown}
The original example of a Brown category is the homotopy category of spectra, 
shown in \cite{brown} and then more generally \cite{adams}. This was generalised by Neeman in \cite{neemantba}, to show that any compactly generated triangulated category with $\T^{\c}$ triangle equivalent to a category with a countable skeleton is also a Brown category. This includes, for example $\D(R)$ for $R$ a countable ring. Necessary and sufficient conditions for $\msf{StMod}(kG)$ to be Brown were determined in \cite[Theorem 4.7.1]{BenGna}.
\end{rem}

We shall consider failures of the derived category to be Brown, and the history of these negative results, in the next section.

\begin{slab}[Neeman's criterion]\label{neemancrit}
For our purposes, the most useful equivalent condition of $\T$ being a Brown category is due to Neeman \cite[Propositions 4.10, 4.11]{neemantba}, and Beligiannis \cite[Theorem 11.8]{belrhap} . This states that
\begin{quote}
\centering
$\T$ is a Brown category if and only if for all $f\in\Flat{\T^{\c}}$ one has $\msf{pd}(f)=1=\msf{id}(f)$, that is $\t{Ext}_{\T^{\c}}^{2}(f,-)=0=\t{Ext}_{\T^{\c}}^{2}(-,f)$.
\end{quote}
\end{slab}

We now show the following, which is our most general result about the functoriality of Brown--Adams representability.

\begin{thm}\label{thm:BAfunctoriality}
Let $\T$ and $\U$ be compactly generated triangulated categories and $F\colon\Flat{\T^{\c}}\to\Flat{\U^{\c}}$ a definable functor.
\begin{enumerate}
\item Suppose $F$ is essentially surjective up to direct summands, then if $\T$ is a Brown category so is $\U$.
\item Suppose $F$ is fully faithful, then if $\U$ is a Brown category so is $\T$.
\end{enumerate}
\end{thm}

\begin{proof}
For the first statement, we may extend $F$ to a definable functor $\bar{F}\colon\Mod{\T^{\c}}\to\Mod{\U^{\c}}$ which admits a left adjoint $\Lambda$, by \cite[Proposition 4.12]{bwdeffun}, hence we are in the setting of \cref{prop:pextisos}. Let $f\in\Flat{\T^{\c}}$ and $g\in\Mod{\U^{\c}}$. Then there are isomorphisms $\t{Ext}_{\T^{\c}}^{n}(\Lambda(g),f)\simeq\t{Ext}_{\U^{\c}}^{n}(g,F(f))$ for all $n\geq 0$, so if $\T$ is Brown, we have $\t{Ext}_{\U^{\c}}^{2}(g,F(f))=0$ by \cref{neemancrit}. The assumption of $F$ being essentially surjective up to direct summands now gives that $\t{Ext}_{\U^{\c}}^{2}(-,h)=0$ for all $h\in\Flat{\U^{\c}}$, hence $\U$ is Brown.

For the second statement, let $m\in\Mod{\T^{\c}}$ and consider the short exact sequence $0\to k\to h\to m\to 0$ in $\Mod{\T^{\c}}$, where $h\in\Flat{\T^{\c}}$ and $k\in\msf{Cot}(\T^{\c})$, which exists by \cref{ex:flatcotpair}. If $f\in\Flat{\T^{\c}}$, then applying $\Hom(f,-)$ to this exact sequence, and noting that $\t{Ext}_{\T^{\c}}^{i}(f,k)=0$ for all $i\geq 0$, gives isomorphisms $\t{Ext}_{\T^{\c}}^{i}(f,e)\simeq\t{Ext}_{\T^{\c}}^{i}(f,m)$ for all $i\geq 2$. As such, we see that 
\[
\t{Ext}_{\T^{\c}}^{2}(f,-)=0 \t{ iff }\t{Ext}_{\T^{\c}}^{2}(f,g)=0 \t{ for all }g\in\Flat{\T^{\c}},
\]
and as such, to determine if $\T$ is a Brown category, it suffices to show that $\t{Ext}_{\T^{\c}}^{2}(f,g)=0$ for all flat functors $f,g\in\Flat{\T^{\c}}$.

As $F$ is definable, it is exact on $\Flat{\T^{\c}}$, and thus, as it is also fully faithful, it induces an injective map $\t{Ext}_{\T^{\c}}^{2}(f,g)\to\t{Ext}_{\U^{\c}}^{2}(Ff,Fg)$, just by considering Yoneda extensions. As such, if $\t{Ext}_{\U^{\c}}^{2}(Ff,Fg)=0$, which occurs whenever $\U$ is Brown, we also have $\t{Ext}_{\T^{\c}}^{2}(f,-)=0$, which finishes the proof.
\end{proof}

\begin{ex}
\cref{thm:BAfunctoriality} can be used to give an alternative proof that if $\T$ is a Brown category and $\msf{L}$ is a smashing subcategory of $\T$, then the category $\T/\msf{L}$ is also a Brown category. Note that a localising subcategory $\msf{L}$ is smashing if and only if the right adjoint $R\colon\msf{T}/\msf{L}\to \msf{L}$ to the localisation functor preserves coproducts, which is equivalent to it being a triangulated definable functor, see \cite[Proposition 4.18]{bwdeffun}. 

In this setting the induced definable functor $F\colon\Mod{(\T/\msf{L})^{\c}}\to \Mod{\T^{\c}}$ corresponding to $R$ is fully faithful - this is because we may identify $F$ with restriction along $L\vert\colon\T^{\c}\to (\T/\msf{L})^{\c}$, where $L$ is the smashing localisation functor. It follows, see for instance \cite[Remark 4.3]{krsmash}, that $F$ is fully faithful. In particular, if $\T$ is a Brown category, we deduce that $\T/\msf{L}$ is, as claimed.

This is, of course, not a new result. A detailed account can be found in \cite[\S 4.3]{axiomatic}, and it can also be found in \cite[\S 11.4]{belrhap}.
\end{ex}

\section{Pure global dimensions of derived categories}
In this section we will use some of the already established results to consider the pure global dimension of derived categories of rings, and give some applications as to determining whether $\D(R)$ is a Brown category.

Our first result establishes a relationship between the pure global dimension of a module category and its derived category.

\begin{prop}\label{prop:pgdimineq}
Let $R$ be a ring. Then $\msf{pgdim}(\D(R))\geq \msf{pgdim}(R)$.
\end{prop}
\begin{proof}
Consider the fully faithful functor $i\colon R\to \D(R)^{\c}$ which, by \cref{prel:modules}, induces an adjoint triple
\[
\begin{tikzcd}
\Mod{\D(R)^{\c}} \arrow[r, "i^{*}" description] \arrow[r, leftarrow, shift left = 2ex, "i^{!}"] \arrow[r, leftarrow, shift right = 2ex, swap, "i_{*}"]& \Mod{R}.  
\end{tikzcd}
\]
Since $i$ is fully faithful, so is $i^{!}$, and thus so is $i_{*}$, in other words the above diagram is the right hand side of a recollement. Now, consider $X\in\Mod{R}$, then viewing $X$ as an object of $\D(R)$, we obtain a functor $\y X\in\Flat{\T^{\c}}$. Applying \cref{prop:pextisos}, we obtain isomorphisms
\[
\t{PExt}_{R}^{n}(M,i^{*}\y X)\simeq \t{PExt}_{\D(R)^{\c}}^{n}(i^{!}M,\y X)\simeq \t{Ext}_{\D(R)^{\c}}^{n}(i^{!}M,\y X),
\]
where the last isomorphism holds because $\y X\in\Flat{\T^{\c}}=\msf{fpInj}(\T^{\c})$. Yet since $i^{*}(F)=F(R)$ for all $F\in\Mod{\T^{\c}}$, we have that $i^{*}(\y X)=\Hom_{\D(R)}(R,X)\simeq X$, and thus $\t{PExt}_{R}^{n}(M,X)\simeq\t{Ext}_{\D(R)^{\c}}^{n}(i^{!}M,\y X)$. Consequently, we see that $\msf{pid}(X\in\Mod{R})\leq \msf{id}(\y X)=\msf{pid}(X\in\D(R))$, and so $\msf{pgdim}(R)\leq \msf{pgdim}(\D(R))$.
\end{proof}

\begin{rem}
There are several remarks to make in relation to \cref{prop:pgdimineq}. 
Beligiannis first illustrated this inequality for hereditary rings in \cite[Proposition 12.8]{belrhap}. This was improved upon by Keller, Neeman and Christensen in \cite[Proposition 1.4]{ckn}, where the assumptions placed on $R$ being that it is coherent, and that every finitely presented $R$-module has finite projective dimension. 
\end{rem}

\begin{ex}[Keller and Neeman's example]
In \cite[\S 6]{neemantba}, Neeman presents a counterexample credited to Keller showing that $\D(R)$ need not be a Brown category. It is shown that if $R$ is a ring such that there is a flat $R$-module of (pure) projective dimension at least $2$, then $\D(R)$ cannot be a Brown category. Such a ring and module are $k[x,y]$ and $k(x,y)$, where $k$ is an uncountable field. 

A similar argument to the above also recovers this counterexample: note that $i^{!}$ preserves flat modules, and thus $\t{Ext}_{R}^{n}(M,i^{*}F)\simeq\t{Ext}_{\D(R)^{\c}}^{n}(i^{!}M,F)$ for any $M\in\Flat{R}$ and $F\in\Mod{\D(R)^{\c}}$, and as such, if $M$ does not have projective dimension at most one, $\D(R)$ is not Brown. 
\end{ex}

\begin{ex}[Finite dimensional local algebras]\label{ex:fdalgebras}
Let $k$ be an uncountable algebraically closed field and $A$ a finite dimensional local $k$-algebra. If $A$ is of finite representation type, then it is pure semisimple, meaning every $A$-module is pure injective, and thus $\msf{pgdim}(A)=0$. If $A$ is not of finite representation type, then $\msf{pgdim}(A)\geq 2$ by \cite[Proposition 5.3]{BBL}, and thus $\D(A)$ is not a Brown category by \cref{prop:pgdimineq}. 

It is not clear whether $A$ being pure semisimple implies that $\D(A)$ is Brown without an assumption of $A$ being hereditary. See \cite[\S 12]{belrhap} for this consideration.
\end{ex}

Let us now consider how pure global dimension of the derived category varies over certain ring epimorphisms. Recall that if $f\colon R\to S$ is a ring epimorphism, then it is \emph{homological} if $(-)_{R}\colon\D(S)\to\D(R)$ is fully faithful.

\begin{prop}\label{prop:hringepi}
Let $f\colon R\to S$ be a homological ring epimorphism, then $\msf{pgdim}(\D(S))\leq \msf{pgdim}(\D(R))$.
\end{prop}
\begin{proof}
If $f$ is homological then $\Phi(-):=(-)_{R}\colon\D(S)\to\D(R)$ is a product and coproduct preserving triangulated functor, which is therefore definable by \cite[Proposition 4.18]{bwdeffun}. Let $F\colon\Mod{\D(S)^{\c}}\to\Mod{\D(R)^{\c}}$ be the induced functor from \cref{ex:tridef} extending $\Phi$. First, note that $F$ has both a right and left adjoint, by \cite[Theorem 4.21]{bwdeffun}, and the left adjoint restricts to flat objects as $F$ is exact. Let $M,N\in\Mod{\D(S)^{\c}}$, and write $M\simeq \msf{colim}_{I}\,\y A_{i}$ as a colimit of projective objects and using \cref{prel:bergman}, write $N\simeq \llim_{J} \y X_{j}$ as an inverse limit of injective objects, where $A_{i}\in\D(S)^{\c}$ and $X_{j}$ are pure injective. We claim that $F$ is fully faithful. To prove this, observe there are isomorphisms
\begin{align*}
\Hom(FM,FN)&\simeq \Hom(F(\msf{colim}_{I}\,\y A_{i}), F(\llim_{J}\y X_{i})) 
\\ &\simeq \Hom(\msf{colim}_{I}\, F\y A_{i}, \llim_{J} F\y X_{j}) \\
&\simeq \lim_{I}\llim_{J}\Hom(\y \Phi A_{i},\y \Phi X_{j}) \\
&\simeq \lim_{I}\llim_{J}\Hom(A_{i},X_{j}) \simeq \Hom(M,N),
\end{align*}
where the second row isomorphism follows from the fact $F$ preserves all limits and colimits, the third follows from the fact that $F\y\simeq \y f_{*}$, the fourth row from the fact that both $\y$ and $f_{*}$ are fully faithful on these objects since $X_{j}$ is pure injective, as discussed in \cref{prel:puretri}. This proves the fully faithfulness of $F$, as claimed. As such, $F\colon\Flat{\D(S)^{\c}}\to\Flat{\D(R)^{\c}}$ is fully faithful and definable with a left adjoint, from which we apply \cref{lem:localisation} to deduce that $\msf{pgdim}(\D(S)):=\msf{pgdim}(\Flat{\D(S)^{\c}})\leq \msf{pgdim}(\Flat{\D(R)^{\c}})=\msf{pgdim}(\D(R))$, which finishes the proof.
\end{proof}

\subsection{von Neumann regular rings control Brown--Adams representability}

We now turn our attention to illustrating that, for a given ring, whether or not its derived category is a Brown category can be determined by ring theoretic properties of a von Neumann regular ring. To do this, we recall the relevant definitions, and investigate the derived categories, of von Neumann regular rings.

\begin{slab}[von Neumann regular rings]
A ring $R$ is \emph{von Neumann regular}, also called \emph{absolutely flat}, if for any $x\in R$ there is a $y\in R$ such that $x=xyx$; this is equivalent to $\Mod{R}=\Flat{R}$, or equivalently $\mod{R}=\msf{proj}(R)$. Any von Neumann regular ring is coherent on both sides, see \cite[Corollary 2.3.23]{psl}.

One can also provide a functorial description of von Neumann regular rings: $R$ is von Neumann regular if and only if $\mod{R}\simeq\mc{A}(R)$, where $\mc{A}(R)=\msf{fp}(\msf{fp}(R,\ab),\ab)$ is the free abelian category on $R$, see \cite[Proposition 10.2.38]{psl}. It is not difficult to see, simply by unravelling the definition, that $R$ is von Neumann regular if and only if the functor $d_{R}\colon\Flat{R}\to\mbb{D}(\Flat{R})$ of \cref{prel:catDA} is an equivalence.
\end{slab}

Before returning to the question of Brown categories, we consider some properties of the derived category of von Neumann regular rings, some of which will be useful for that question. 

\begin{slab}[A generating localisation]\label{freydloc}
Let $R$ be any ring and consider the subset of $\D(R)^{\c}$ given by $\scr{R}:=\{\Sigma^{i}R:i\in\Z\}$, and view it as a small preadditive category. As in \cref{prel:modules}, the fully faithful inclusion $i\colon \scr{R}\to\D(R)^{\c}$ induces a recollement
\[
\begin{tikzcd}[column sep = 1.5cm]
\scr{S} \arrow[r, hook]& \Mod{\D(R)^{\c}} \arrow[l, two heads, shift left = 1.5ex] \arrow[l, two heads, shift right = 1.5ex] \arrow[r, two heads, "Q" description] & \Mod{\scr{R}} \arrow[l, shift right = 1.5ex, hook']\arrow[l, shift left = 1.5ex, hook', "\rho"]
\end{tikzcd}
\]
where $Q(f)=f\circ i=f\vert_{\scr{R}}$, and $\scr{S}=\msf{ker}(Q)=\{f\in\Mod{\T^{\c}}:f(\Sigma^{i}R)=0 \t{ for all }i\in\Z\}$. Observe that since $\Hom_{\D(R)}(\Sigma^{i}R,\Sigma^{j}R)\simeq\delta_{i,j}R$, there is an equivalence $\Mod{\scr{R}}\simeq\prod_{\Z}\Mod{R}$, see \cite[\S 5]{gpzg}.

We call $Q$ the \emph{Freyd generating localisation}. To justify this, recall that Freyd's generating hypothesis for rings states \begin{quote}
If $f\colon C\to D$ is a morphism of perfect complexes such that
\begin{equation}\label{fgh}
0= \Hom_{\D(R)}(\Sigma^i R, f)\colon \Hom_{\D(R)}(\Sigma^i R,C)\to \Hom_{\D(R)}(\Sigma^i R,D)
\end{equation}
for all $i\in\Z$, then $f=0$.
\end{quote} 
\end{slab}

The justification for its naming is that the kernel of the generating localisation determines the holding of the generating hypothesis, as the next result shows.

\begin{prop}\label{prop:fghholds}
The generating hypothesis holds for $\D(R)$ if and only if $\scr{S}\cap\mod{\D(R)^{\c}}=0$. In particular, if $R$ is right coherent, then $\D(R)$ satisfies the generating hypothesis if and only if $R$ is von Neumann regular.
\end{prop}

\begin{proof}
Let $\alpha\colon A\to B$ be a morphism in $\D(R)^{\c}$. Since $\y R[i]$ is projective in $\Mod{\D(R)^{\c}}$ for all $i\in\Z$, and there are isomorphisms $\Hom(\y R[i],\y A)\simeq \Hom_{\D(R)}(R[i],A)$, one obtains an isomorphism $\Hom(\y R[i],\msf{Im}\,\y\alpha)\simeq\msf{Im}\,\Hom_{\D(R)}(R[i],\alpha)$. Consequently, we see that $\msf{Im}\,\y\alpha\in\scr{S}$ if and only if $\Hom_{\D(R)}(R[i],\alpha)=0$ for all $i\in\Z$.

Therefore, if the generating hypothesis holds, we see that $\scr{S}\cap\mod{\D(R)^{\c}}=0$. On the other hand, if $\scr{S}\cap\mod{\D(R)^{\c}}=0$, the above isomorphism tells us that $\msf{Im}(\y \alpha)=0$, hence $\alpha=0$, showing the generating hypothesis holds. This proves the first statement.

The localisation functor $Q$ is of finite type if and only if $R$ is right coherent by \cite[Theorem 5.1]{gpinj}, and thus $\scr{S}=\rlim(\scr{S}\cap\mod{\D(R)^{\c}})$. This means that $Q$ is an equivalence if and only if $\scr{S}\cap\mod{\D(R)^{\c}}=0$, which we just saw is equivalent to the generating hypothesis holding. Yet $Q$ is an equivalence if and only if $R$ is von Neumann regular, as shown in \cite[Theorem 8.2]{gpzg}.
\end{proof}

\begin{rem}\label{rem:sgh}
 Notice we also immediately obtain that the \emph{strong generating hypothesis} holds, that is $H_{*}\colon\D(R)^{\c}\to\Mod{R}$ is fully faithful, for a right coherent ring $R$ if and only if $R$ is von Neumann regular. Indeed, if $R$ is von Neumann regular, then the Freyd localisation $Q\colon\Mod{\D(R)^{\c}}\to\prod_{\Z}\Mod{R}$, is an equivalence. Its left adjoint $i^{*}\colon\prod_{\Z}\Mod{R}\to\Mod{\D(R)^{\c}}$ is given by $(M_{i})_{\Z}\mapsto \y (\oplus_{\Z}M_{i}[-i])$, in view of the fact that it preserves flat objects and sends $R_{i}\mapsto \y R[-i]$. Since $\y\colon\D(R)\to\Mod{\D(R)^{\c}}$ is fully faithful when the first component is compact, as discussed in \cref{prel:puretri}, we deduce that $\Hom_{\D(R)}(P,Q)\simeq \Hom(\y P,\y Q)\simeq \Hom(i^{*}(\oplus_{\Z}H_{i}(P)[-i]),i^{*}(\oplus_{\Z}H_{j}(Q)[-j]))\simeq \Hom_{R}(H_{*}P,H_{*}Q)$.

 Of course, neither of these results is new. That the generating hypothesis holds for a right coherent ring if and only if the ring is von Neumann regular is \cite[Corollary 3.2]{HLP}, while the fact that the strong generating hypothesis holds if and only if the generating hypothesis holds is \cite[Theorem 1.3]{HLP}. However, the functor categorical approach via the generating localisation differs from, and is much shorter than, the original proofs given at the references.
\end{rem}

Let us say a bit more about derived categories of von Neumann regular rings. Recall from \cite{krcqsl} that, for a compactly generated triangulated category $\T$, a Serre subcategory $\mc{S}\subseteq\mod{\T^{\c}}$ is \emph{perfect} if the right adjoint $\rho_{\mc{S}}\colon\Mod{\T^{\c}}/\rlim\mc{S}\to\Mod{\T^{\c}}$ to the localisation $Q_{\mc{S}}\colon\Mod{\T^{\c}}\to\Mod{\T^{\c}}/\rlim\mc{S}$ is exact.

\begin{prop}\label{prop:telescope}
Let $R$ be a von Neumann regular ring, then every Serre subcategory of $\mod{\D(R)^{\c}}$ is perfect. Consequently, there is a bijection between
\begin{enumerate}
\item thick subcategories of $\D(R)^{\c}$;
\item Serre subcategories of $\mod{R}$;
\item $\Sigma$-invariant Serre subcategories of $\mod{\D(R)^{\c}}$;
\item smashing localisations of $\D(R)$;
\end{enumerate}
and all these are in bijection with two-sided ideals of $R$. In particular, the telescope conjecture holds in $\D(R)$.
\end{prop}

\begin{proof}
As $R$ is von Neumann regular if and only if the Freyd localisation $Q$ is an equivalence, we see that, since every finitely presented object of $\Mod{R}$ is projective, every finitely presented object of $\Mod{\D(R)^{\c}}\simeq\prod_{\Z}\Mod{R}$ is also projective. This is because an object in this category is projective if and only if each component is projective. As such, if $\mc{S}\subseteq\mod{\D(R)^{\c}}$, we have $\t{Ext}_{\D(R)^{\c}}^{2}(\mc{S},-)=0$ which means $\rho_{\mc{S}}$ is exact via the identification $\rho_{\mc{S}}\colon\Mod{\D(R)^{\c}}/\rlim\mc{S}\xrightarrow{\simeq} \mc{S}^{\perp_{0,1}}$.

We now establish the bijections. Over any coherent ring, if $\msf{L}\subset\D(R)^{\c}$ is thick, then $\mc{W}_{\msf{L}}=\{H_{0}(X):X\in\msf{L}\}$ is a wide subcategory of $\mod{R}$, and every wide subcategory of $\mod{R}$ arises this way, see \cite{Hovey}. As $R$ is von Neumann regular, if $X\in\D(R)$ then $X\simeq \oplus_{i\in\Z}H_{i}(X)[-i]$ by \cite[Theorem 8.2]{gpzg}, hence $\msf{L}=\{X:H_{i}(X)\in\mc{W}_{\msf{L}} \t{ for all }i\in\Z\}$, which establishes a bijection between thick subcategories of $\D(R)^{\c}$ and wide subcategories of $\mod{R}$. Yet as $R$ is von Neumann regular every short exact sequence in $\mod{R}$ is split, and thus a subcategory of $\mod{R}$ is wide if and only if it is Serre. This give $(1)\iff (2)$.

The bijection between Serre subcategories of $\mod{R}$ and $\Sigma$-invariant Serre subcategories of $\mod{\D(R)^{\c}}$ is immediate from the equivalence $\Mod{\D(R)^{\c}}\simeq\prod_{\Z}\Mod{R}$. In particular, any $\Sigma$-invariant Serre subcategory of $\mod{\D(R)^{\c}}$ is $(\mc{S})_{i\in\Z}$, for $\mc{S}\subset\mod{R}$. This gives $(2)\iff(3)$

The bijection between $\Sigma$-invariant perfect Serre subcategories and smashing localisations is due to Krause \cite[\S 12]{krcqsl}. We have seen every Serre subcategory of $\mod{\D(R)^{\c}}$ is perfect, so this bijection is really between $\Sigma$-invariant Serre subcategories and smashing localisations. This gives $(3)\iff (4)$.

The bijection with two-sided ideals of $R$ is found at \cite[Theorem 4.6]{herzog}; and the telescope conjecture is precisely the equivalence of $(1)$ and $(4)$.
\end{proof}

The proof of $(1)\iff (2)$ in the above proposition actually holds in more generality, noting that homology preserves all coproducts and products. Consequently, we also have the following.

\begin{cor}\label{prop:loccolocbij}
Let $R$ be a von Neumann regular ring. Then the assignment $\msf{L}\mapsto H_{0}(\msf{L})$ gives a bijection between localising (resp. colocalising) subcategories of $\D(R)$ and wide subcategories of $\Mod{R}$ closed under coproducts (resp. products).
\end{cor}

We not that the telescope conjecture holding for $\D(R)$ for $R$ von Neumann regular was proved, using completely different methods, in \cite[Corollary 3.12]{bazzhrb}.

\begin{ex}
Let us illustrate how \cref{prop:telescope} can be used to completely describe smashing localisations of $\D(R)$ for $R$ von Neumann regular. Let $\mc{S}\subseteq\mod{R}$ be a Serre subcategory. Then, by \cite[Theorem 4.6]{herzog} there is a two sided ideal $J\subseteq R$ such that $\mc{S}=\{P\in\mod{R}:P\in\msf{add}(J)\}$. It is clear that the localisation functor associated to $\mc{S}$ is given by extension along $R\to R/J$, and so we may associate $\mod{R/J}\simeq\mod{R}/\mc{S}$. Since $R/J$ is von Neumann regular, there is an equivalence of categories $\Mod{\D(R/J)^{\c}}\simeq\prod_{i\in\Z}\Mod{R/J}$.

Now, let $\scr{L}$ be a smashing localisation of $\D(R)$; then, as we have established the telescope conjecture holding, there is a two sided ideal $J\subseteq R$ such that
$\scr{L}$ consists of the objects of $\D(R)$ that vanish under the composition
\[
\D(R)\xrightarrow{\y}\Mod{\scr{P}}\simeq \prod_{\Z}\Mod{R}\xrightarrow{\prod_{\Z}(-\otimes R/J)}\prod_{\Z}\Mod{R/J}.
\]
Yet these are none other than the objects of $\D(R)$ which vanish under $-\otimes_{R}R/J\colon\D(R)\to\D(R/J)$. In other words, $\D(R)/\scr{L}\simeq\D(R/J)$.
\end{ex}

The following illustrates that the proof strategy of \cref{prop:telescope} is not generalisable beyond von Neumann regular rings, and shows that von Neumann regular rings are uniquely determined by the kernel of the Freyd localisation being perfect.

\begin{prop}
Let $R$ be a right coherent ring. Then the following are equivalent:
\begin{enumerate}
    \item Every $\Sigma$-invariant definable subcategory of $\D(R)$ is triangulated.
    \item Every $\Sigma$-invariant Serre subcategory of $\mod{R}$ is perfect;
    \item The Freyd subcategory $\scr{S}\cap\mod{\D(R)^{\c}}=\{f\in\mod{\D(R)^{\c}}:f(\Sigma^{i}R)=0 \t{ for all }i\in\Z\}$ is perfect;
    \item $R$ is von Neumann regular;
\end{enumerate}
\end{prop}
\begin{proof}
The crucial fact that we shall use, that is the substance of \cite[\S 12]{krcqsl}, is that a definable subcategory $\mc{D}\subseteq\D(R)$ is triangulated if and only if $\mc{S}(\mc{D})\subseteq\mod{\D(R)^{\c}}$ is perfect and $\Sigma$-invariant. This fact quite quickly deals with the equivalence $(1)\iff (2)$. Indeed, the bijection of \cref{prel:definable} restricts to bijections between $\Sigma$-invariant (respectively triangulated) definable subcategories of $\D(R)$ and $\Sigma$-invariant (respectively $\Sigma$-invariant perfect) Serre subcategories of $\mod{\D(R)^{\c}}$, which proves the equivalence.

Since the Freyd subcategory $\scr{S}\cap\mod{\D(R)^{\c}}$ is $\Sigma$-invariant, the implication $(2)\implies (3)$ is trivial, while $(4)\implies (2)$ is \cref{prop:telescope}. The proof is therefore complete by showing the $(3)\implies (4)$ implication. If the Freyd subcategory $\scr{S}\cap\mod{\D(R)^{\c}}$ is perfect, then, as it is also $\Sigma$-invariant, we have $\mc{D}(\scr{S}\cap\mod{\D(R)^{\c}})$ being a triangulated subcategory of $\D(R)$. Yet $\mc{D}(\scr{S}\cap\mod{\D(R)^{\c}})$ is the smallest definable subcategory containing $\scr{R}$ - indeed, if $\mc{D}'$ is another definable subcategory containing $\scr{R}$, then $\mc{S}(\mc{D}')\subseteq\scr{S}\cap\mod{\D(R)^{\c}}$ by the bijection in \cref{prel:definable}; as the bijection is order reversing, it follows that $\mc{D}(\scr{S}\cap\mod{\D(R)^{\c}})\subseteq\mc{D}'$. Yet since the smallest triangulated subcategory closed under coproducts and containing $R$ is $\D(R)$, it follows that $\D(R)$ is the smallest definable subcategory containing $\scr{R}$ and this is equivalent to Freyd's generating hypothesis, see \cite{krcfsh}, which, as we showed in \cref{prop:fghholds}, is equivalent to $R$ being von Neumann regular.
\end{proof}

We are also able to describe ideals of morphisms in $\D(R)^{\c}$ for von Neumann regular rings. 

\begin{slab}[Cohomological ideals of compacts]    
Recall that an ideal $\mf{I}\subseteq \D(R)^{\c}$ is called \emph{cohomological} if there is a cohomological functor $H\colon\D(R)^{\c}\to\scr{A}$ into an abelian category such that $\mf{I}=\{\phi:H(\phi)=0\}$. The map $\mf{I}\mapsto \msf{Im}(\mf{I})=\{f\in\mod{\T^{\c}}:f\simeq \msf{Im}(\y\alpha) \t{ for some }\alpha\in\mf{I}\}$ establishes a bijection between Serre subcategories of $\mod{\D(R)^{\c}}$ and cohomological ideals of $\D(R)^{\c}$, see \cite[\S 8]{krcqsl}.

An ideal $\mf{I}\subseteq\D(R)^{\c}$ is \emph{exact} if it is cohomological and satisfies $\mf{I}^{2}=\mf{I}$ and $\Sigma \mf{I}=\mf{I}$; there is a bijection between smashing subcategories and exact ideals of $\T^{\c}$, see \cite[\S12]{krcqsl}.
\end{slab}

For the telescope conjeture, cohomological ideals generated by identity maps are fundamental.

\begin{lem}\label{lem:ciregular}
Let $\T$ be a compactly generated triangulated category, and suppose $\mf{I}$ is a cohomological ideal of $\T^{\c}$. If $\mf{I}$ is generated by identity maps then $\msf{Im}(\mf{I})$ is generated by projective objects. If $\T=\D(R)$ for $R$ von Neumann regular, then the converse holds, that is $\msf{Im}(\mf{I})$ is generated by projective objects if and only if $\mf{I}$ is generated by identity maps.
\end{lem}

\begin{proof}
First, recall that if $\phi_{1},\phi_{2}\in\Hom_{\T^{\c}}(A,B)$, then the sum $\phi_{1}+\phi_{2}$ is given by the composition $A\xrightarrow{\Delta}A\oplus A\xrightarrow{\phi_{1}\oplus\phi_{2}}B\oplus B\xrightarrow{\nabla}B$, where $\Delta$ is the diagonal map and $\nabla$ is the codiagonal. Therefore, if, for both $i=1,2$, the map $\phi_{i}$ factors through $\t{Id}_{C_{i}}$ for some $C_{i}\in\T^{\c}$, we see that $\phi_{1}+\phi_{2}$ factors through $\t{Id}_{C_{1}\oplus C_{2}}$. 

As such, by induction, if $f\simeq\msf{Im}(\y\phi)$ for some $\phi\in \mf{I}$ and $\mf{I}$ is generated by identity maps, we may assume that $\phi=\alpha\circ\t{Id}_{A}\circ\beta$ for some $A\in\T^{\c}$. In this case, we have $f\simeq \msf{Im}(\y(\alpha\circ\t{Id}_{A}\circ\beta))$, and this is a subobject of $\msf{Im}(\y(\alpha\circ\t{Id}_{A}))$; furthermore $\msf{Im}(\y(\alpha\circ\t{Id}_{A})$ is a quotient of $\msf{Im}(\y\t{Id}_{A})$, which is a projective object. This shows that $f$ is a subquotient of a projective object, and thus $\msf{Im}(\mf{I})$ is generated by projectives.

Let $R$ be a von Neumann regular ring and suppose that $\gamma\colon A\to B$ is a morphism in $\mf{I}$. There is then a factorisation $\y A\twoheadrightarrow \msf{Im}(\y\gamma)\hookrightarrow\y B$ of $\y \gamma$ in $\mod{\D(R)^{\c}}$, but we have seen that every object in $\mod{\D(R)^{\c}}$ is projective, and thus $\msf{Im}(\y\gamma)\simeq \y C$ for some $C\in\D(R)^{\c}$. By the fullyfaithfullness of $\y$ on compact objects, it follows that $\gamma$ factors through $\t{Id}_{C}$, so $\mf{I}$ is generated by identity morphisms.
\end{proof}

\begin{rem}
We get another proof that the telescope conjecture holds for derived categories of von Neumann regular rings from \cref{lem:ciregular}. Indeed, the telescope conjecture is equivalent to the statement 
\begin{quote}
\centering
every exact ideal is generated by identity maps    
\end{quote}
see \cite[Theorem 13.4]{krcqsl}, but every Serre subcategory of $\mod{\D(R)^{\c}}$ is generated by projective objects, so this statement holds immediately. Note that all ideals of $\D(R)^{\c}$, for $R$ von Neumann regular, can be defined in terms of ideals of $\mod{R}$ via the strong generating hypothesis \cref{rem:sgh}.
\end{rem}

Let us now return to the question of Brown--Adams representability. To begin with, we recall a classic construction of Olivier, and a more recent generalisation of it due to Herzog and L'Innocente.

\begin{slab}[von Neumann approximations of rings]
In \cite[Proposition 5]{olivier}, Olivier demonstrated that the inclusion of the category of commutative von Neumann regular rings inside the category of commutative rings admits a left adjont, and as such corresponding to any commutative ring $A$ is a unique commutative von Neumann ring $T(A)$. Explicitly,
\[
T(A)=A[x_{a}:a\in A]/(x_{a}^{2}a-x_{a}, x_{a}a^{2}-a)_{a\in A}.
\]
Furthermore, the map $A\to T(A)$ given by the unit of the adjunction is an epimorphism of rings.

This construction was generalised to all rings in \cite[Theorem 2.2]{herzinn}. More precisely, it is shown that for a ring $R$ there is a unique abelian von Neumann regular ring $R^{\t{ab}}$ with a ring epimorphism $R\to R^{\t{ab}}$ that is universal with respect to this property. Unlike in the commutative case, $R^{\t{ab}}$ can be the zero ring; see \cite[Corollary 2.3]{herzinn} for necessary and sufficient conditions for this to occur.
\end{slab}

We now have the following result.

\begin{prop}\label{prop:vnrdeterminesbrown}
Let $R$ be a ring. If $R^{\mrm{ab}}\neq 0$, then if $R^{\mrm{ab}}$ is not hereditary, $\D(R)$ is not a Brown category.
\end{prop}

\begin{proof}
Since $R\to R^{\t{ab}}$ is an epimorphism, we have $\msf{pgdim}(R^{\t{ab}})\leq \msf{pgdim}(R)$ by \cref{prop:ringepi}, and we further have $\msf{pgdim}(R)\leq \msf{pgdim}(\D(R))$ by \cref{prop:hringepi}. Since, as $R$ is von Neumann regular, we have $\msf{pgdim}(R)=\msf{gdim}(R)$, we see that if $R$ is not hereditary it cannot be the case that $\D(R)$ is Brown, by \cref{neemancrit}.
\end{proof}

\begin{ex}
Clearly then, one would like to know when $R^{\t{ab}}$ is, or is not, hereditary. In the case that $R$ is commutative, which ensures that $R^{\t{ab}}\neq 0$, a result of Kie\l pi\'{n}ski and Simson (see \cite{HZ} or \cite{KS}) states that if $\t{card}(R)=\aleph_{m}$ and $X$ is a set of cardinality $\aleph_{n}$, then $\msf{pgdim}(R[X])=1+\msf{max}\{m,n\}$. Consequently, by \cref{prop:ringepi}, we can bound $\msf{pgdim}(R^{\t{ab}})=\msf{gdim}(R^{\t{ab}})$ above by $1+n$, where $\aleph_{n}=\msf{card}(R)$. As such, whenever $R$ is countable, one has that $R^{\t{ab}}$ is hereditary. This can also be deduced from \cite[Proposition 10.5]{GJdim}. However, this is not saying much: if $R$ is countable, then $\msf{perf}(R)$ has a countable skeleton, hence $\D(R)$ is a Brown category by Neeman's result discussed in \cref{rem:exofbrown}. 

Of course, if one knows the pure global dimension of $\Mod{R}$, one need not even make the passage to $R^{\t{ab}}$. Examples of rings whose pure global dimension is known can be found in \cite[\S 11]{JL}, and a survey of pure global dimension can be found in \cite{HZ}.
\end{ex}

\begin{chunk}\label{schemes}
Let us now demonstrate how the above results can be applied to schemes. Let $X$ be a quasi-compact quasi-separated (qcqs) scheme, which by  \cite[Theorem 3.1.1]{BvdB} ensures that the category $\D_{\t{qc}}(X)$, the derived category of complexes of $\scr{O}_{X}$-modules having quasi-coherent cohomology, is a compactly generated triangulated category. The compact objects are the perfect complexes $\msf{perf}(X)$.

Let $f\colon U\to X$ be the inclusion of a quasi-compact open subscheme. The assumption that $X$ is qcqs means, by combining \cite[Proposition 3.9.2]{lipman} and Grothendieck duality \cite[Theorem 4.1]{lipman} (or \cite[Lemma 48.3.1]{stack}) that there is an adjoint triple $\textbf{L}f^{*}\dashv\textbf{R}f_{*}\dashv f^{\times}$ where $\textbf{R}f_{*}\colon\D_{\t{qc}}(U)\to\D_{\t{qc}}(X)$ is fully faithful. In particular, $\textbf{R}f_{*}$ is a fully faithful definable triangulated functor. 
\end{chunk}

\begin{prop}\label{prop:brownschemes}
Let $X$ be a qcqs scheme. Then $\D_{\mrm{qc}}(X)$ is a Brown category only if $\msf{pgdim}(\scr{O}_{X,x})\leq 1$, or $\scr{O}_{X,x}^{\mrm{ab}}$ is a hereditary von Neumann regular ring, for all $x\in X$.
\end{prop}

\begin{proof}
Let $U\simeq\t{Spec}(A)\subseteq X$ be an affine subscheme. There are then inequalities, for any $\mf{p}\in\t{Spec}(A)$,
\[
\msf{pgdim}(\D_{\t{qc}}(X))\geq \msf{pgdim}(\D_{\t{qc}}(U))\geq \msf{pgdim}(A)\geq \msf{pgdim}(A_{\mf{p}})\geq \msf{pgdim}(A_{\mf{p}}^{\t{ab}}),
\]
where the first is by the discussion in \cref{schemes} combined with \cref{thm:BAfunctoriality}, the second is by \cref{prop:pgdimineq}, and the last two are by \cref{prop:ringepi}. Thus if $\D_{\t{qc}}(X)$ is a Brown category, we have $1\geq \msf{pgdim}(A_{\mf{m}})\geq \msf{pgdim}(A_{\mf{m}}^{\t{ab}})$. As any $x\in X$ lies in an affine open, the result now follows immediately.
\end{proof}

\subsection{Beyond derived categories}\label{sec:beyondderived}

One can perform similar arguments to the above in arbitrary compactly generated triangulated categories, in order to determine whether or not they satisfy Brown--Adams representability. The following demonstrates how von Neumann regular rings can be used in this generality.

\begin{prop}
Let $\T$ be a compactly generated triangulated category. If there is a faithful definable functor $F\colon \D(R)\to \T$ with $R$ von Neumann regular, then if $R$ is not hereditary, $\T$ is not Brown.
\end{prop}
\begin{proof}
As $F$ is definable it extends to a definable functor $\hat{F}\colon\Flat{\D(R)^{\c}}\to\Flat{\T^{\c}}$, and if this is faithful we may deduce the conclusion from \cref{thm:BAfunctoriality}. So suppose $M,N\in\Flat{\D(R)^{\c}}$ and write $M\simeq \rlim_{I}\y A_{i}$ for $A_{i}\in\D(R)^{\c}$ and $N\simeq \llim_{J}\y X_{j}$ with $X_{j}\in\msf{Pinj}(\D(R))$, which is just another application of Bergman's result. Since $F$ is fully faithful, there are isomorphisms $\Hom_{\D(R)}(A_{i},X_{j})\hookrightarrow\Hom_{\T}(FA_{i},FX_{j})$ for each $(i,j)\in I\times J$. Now, since $\hat{F}$ is definable, it preserves filtered colimits, but it also preserves limits as it is the restriction of a right adjoint. Consequently, we have $\Hom(\hat{F}M,\hat{F}N)\simeq \llim_{I\times J}\Hom(\hat{F}\y A_{i},\hat{F}\y X_{j})\simeq \llim_{I\times J}\Hom(FA_{i},FXA_{j})$, since $\hat{F}\circ \y \simeq  F$. Consequently the natural map $\Hom(M,N)\simeq \llim_{I\times J}\Hom(\y A_{i},\y X_{j})\to \llim_{I\times J}\Hom(\hat{F}\y A_{i},\hat{F}\y X_{j})\simeq \Hom(\hat{F}M,\hat{F}N)$ is an isomorphism, so $\hat{F}$ is fully faithful.
\end{proof}

Of course, it may not be immediate that such a definable functor exists. For monogenic triangulated categories, we have the following which bypasses such an obstacle.

\begin{thm}\label{thm:monogenic}
Let $\T$ be a monogenic compactly generated triangulated category generated by $S\in\T^{\c}$. If $A=\mrm{End}_{\T}(S)$ has a flat module of projective dimension greater than one, then $\T$ is not a Brown category.
\end{thm}
\begin{proof}
Consider the functor $F\colon \T\to \Mod{A}$ given by $X\mapsto \Hom_{\T}(S,X)$; as $S\in\T^{\c}$, this functor preserves coproducts, products and pure triangles and is thus a coherent functor, in the sense of \cite[\S 3]{bwdeffun}. As such, by \cite[Theorem 3.2]{bwdeffun} there is a unique definable functor $\hat{F}\colon\Flat{\T^{\c}}\to\Mod{A}$ such that $\hat{F}\circ y\simeq F$. However, since $F$ is cohomological, there is also an adjoint triple
\[
\begin{tikzcd}[column sep = 2cm]
\Mod{\T^{\c}} \arrow[r, "\bar{F}" description] & \Mod{A} \arrow[l, shift left = 2ex, hook', "\rho"] \arrow[l, shift right = 2ex, swap, hook', "\Lambda"]
\end{tikzcd}
\]
where $\bar{F}\circ \y\simeq F$ and $\bar{F}\vert_{\Flat{\T^{\c}}}=\hat{F}$. The existence of the triple can either be induced by \cite[Theorem 3.26]{bwdeffun} and the comments following, or from the fact that $\bar{F}$ is nothing other than the restriction along the fully faithful embedding $S\hookrightarrow \T^{\c}$, once one views $S$ as a singleton subcategory with endomorphisms $A$. In particular, this viewpoint illustrates $\Lambda$ and $\rho$ as fully faithful.
Furthermore, $\Lambda$ itself will preserve flat objects, and thus there is an adjoint pair
\[
\begin{tikzcd}
\Flat{\T^{\c}} \arrow[r, shift right = 1ex,swap,  "\hat{F}"]& \Flat{A} \arrow[l, shift right = 1ex,swap, "\Lambda"]
\end{tikzcd}
\]
with $\hat{F}$ a definable functor. Since $\Lambda$ is fully faithful, we may apply \cref{lem:localisation}, and thus $\msf{pgdim}(\T)\geq \msf{pgdim}(\Flat{A})$. The claim now follows immediately from \cref{neemancrit}.
\end{proof}

\begin{ex}
Let $\T$ be a compactly generated tensor-triangulated category such that the tensor unit $\1$ is a compact object. Then the \emph{central ring} of $\T$, $A_{\T}=\t{End}_{\T}(\1)$ is a commutative ring by \cite[Proposition 2.2]{BalmerSSS}. The same proof as for \cref{thm:monogenic} holds to show that if $A_{\T}$ has a flat module of projective dimension at most one, then $\T$ is not a Brown category.
\end{ex}

One may also use cosilting t-structures to provide a lower bound on pure global dimension for an arbitrary compactly generated triangulated category. We refer the reader to \cite{SaoStov} for background on t-structures and their relationship to localisations of the functor category.

\begin{prop}\label{prop:cosilting}
Let $\T$ be a compactly generated triangulated category and suppose that $\mbb{t}$ is a compactly generated t-structure that restricts to compact objects. Then $\msf{pgdim}(\T)\geq\msf{pgdim}(\msf{fpInj}(\scr{H}_{\mbb{t}}))$. If, additionally $\mbb{t}$ is non-degenerate, then $\msf{pgdim}(\T)\geq\msf{pgdim}(\scr{H}_{\mbb{t}})$.
\end{prop}

\begin{proof}
We first make some comments on the assumptions. The fact that $\mbb{t}$ restricts to compacts means that the heart $\scr{H}_{\mbb{t}}$ is a locally coherent category with $\msf{fp}(\scr{H}_{\mbb{t}})$ being generated by $\mrm{H}_{\mbb{t}}^{0}(\T^{\c})$, where both these statements are at \cite[Theorem 8.31]{SaoStov}; in particular, if $A\in\T^{\c}$, then $\mrm{H}_{\mbb{t}}^{0}(A)\in\msf{fp}(\scr{H}_{\mbb{t}})$

Now, let $i\dashv \tau^{\leq 0}\colon \mc{U}\hookrightarrow\T$ denote the inclusion of the aisle of $\mbb{t}$ and its right adjoint, and by setting $\mc{U}_{0}=\mc{U}\cap\T^{\c}$, this adjunction restricts to $i\dashv\tau^{\leq 0}\colon \mc{U}_{0}\hookrightarrow\T^{\c}$. Thus, by \cref{prel:modules}, there is a localisation
\[
\begin{tikzcd}[column sep = 1in]
\Mod{\T^{\c}} \arrow[r, two heads, shift left = 1ex, "i^{*}\simeq(\tau^{\leq 0})^{!}"] & \Mod{\mc{U}_{0}}\arrow[l, shift left = 1ex, hook', "i_{*}\simeq(\tau^{\leq 0})^{*}"]
\end{tikzcd}
\]
where $i_{*}$ is a definable functor which preserves fp-injective objects since $\tau^{\leq 0}$ preserves compact objects.  Let $\mbb{y}\colon\mc{U}\to\Mod{\mc{U}_{0}}$ denote the restricted Yoneda embedding. By \cite[Proposition 8.28]{SaoStov}, there a localisation
\[
\begin{tikzcd}
\Mod{\mc{U}_{0}} \arrow[r, shift left = 1ex, "F"]& \scr{H}_{\mbb{t}} \arrow[l, shift left = 1ex, hook', "G"]
\end{tikzcd}
\]
where $F\circ\mbb{y}\simeq \mrm{H}_{\mbb{t}}^{0}\vert_{\mc{U}}$ and $G=\mbb{y}\vert_{\scr{H}_{\mbb{t}}}$. The functor $F$ preserves finitely presented objects, by \cite[Corollary 8.29]{SaoStov}, meaning that $G$ is a definable functor. Let $X\in\scr{H}_{\mbb{t}}$ and consider the composition $(i_{*}\circ G)(X)$, which, by an explicit computation, is naturally isomorphic to $\Hom_{\mc{U}}(\tau^{\leq 0}(-),X)\simeq\Hom_{\scr{H}_{\mbb{t}}}(\mrm{H}_{\mbb{t}}^{0}(-),X)$, as functors on $\T^{\c}$. As such, we see that if $X\in\msf{fpInj}(\scr{H}_{\mbb{t}})$, then $(i_{*}\circ G)(X)$ is a cohomological functor $(\T^{\c})^{\op}\to\ab$, which is the same as being flat. Consequently, $i_{*}\circ G\colon\msf{fpInj}(\scr{H}_{\mbb{t}})\to\Flat{\T^{\c}}$ is a fully faithful definable functor with a left adjoint, hence we have the first claim by \cref{lem:localisation}.

In the case that $\mbb{t}$ is in addition non-degenerate, we may apply \cite[Theorem 3.6]{amv}, which combines the above localisations in to a single one
\[
\begin{tikzcd}
\Mod{\T^{\c}} \arrow[r, shift left = 1ex, two heads, "\bar{\mrm{H}_{\mbb{t}}^{0}}"]& \scr{H}_{\mbb{t}} \arrow[l, shift left = 1ex, hook',"j_{*}"]
\end{tikzcd}
\]
where $\bar{\mrm{H}_{\mbb{t}}^{0}}\circ \y\simeq \mrm{H}_{\mbb{t}}^{0}$ and $j_{*}\mrm{H}_{\mbb{t}}^{0}(X)\simeq \y X$. In this case, it is clear that $j_{*}\colon\scr{H}_{\mbb{t}}\hookrightarrow\Flat{\T^{\c}}$ is a definable embedding, giving $\msf{pgdim}(\scr{H}_{\mbb{t}})\leq\msf{pgdim}(\Flat{\T^{\c}})=\msf{pgdim}(\T)$.
\end{proof}

We note that a t-strucure that satisfies the conditions of the above proposition must arise from a cosilting object, see \cite[Theorem 3.6]{amv} and \cite[Proposition 9.2]{SaoStov}.
\bibliographystyle{abbrv}
\bibliography{references.bib}

\begin{thebibliography}{10}

\bibitem{adams}
J.~F. Adams.
\newblock A variant of {E}. {H}. {Brown}'s representability theorem.
\newblock {\em Topology}, 10, 1971.

\bibitem{AKLY}
L.~Angeleri~H{\"u}gel, S.~Koenig, Q.~Liu, and D.~Yang.
\newblock Ladders and simplicity of derived module categories.
\newblock {\em J. Algebra}, 472:15--66, 2017.

\bibitem{amv}
L.~Angeleri~H\"ugel, F.~Marks, and J.~Vit\'oria.
\newblock Torsion pairs in silting theory.
\newblock {\em Pacific J. Math.}, 291(2), 2017.

\bibitem{BBL}
D.~Baer, H.~Brune, and H.~Lenzing.
\newblock A homological approach to representations of algebras. {II}: {Tame} hereditary algebras.
\newblock {\em J. Pure Appl. Algebra}, 26:141--153, 1982.

\bibitem{BL}
D.~Baer and H.~Lenzing.
\newblock A homological approach to representations of algebras. {I}: {The} wild case.
\newblock {\em J. Pure Appl. Algebra}, 24, 1982.

\bibitem{BalmerSSS}
P.~Balmer.
\newblock Spectra, spectra, spectra -- tensor triangular spectra versus {Zariski} spectra of endomorphism rings.
\newblock {\em Algebr. Geom. Topol.}, 10(3):1521--1563, 2010.

\bibitem{ban}
A.~Banerjee.
\newblock Some remarks on a theorem of {B}ergman.
\newblock {\em C. R. Math. Acad. Sci. Paris}, 354(7), 2016.

\bibitem{bazzhrb}
S.~Bazzoni and M.~Hrbek.
\newblock Definable coaisles over rings of weak global dimension at most one.
\newblock {\em Publ. Mat., Barc.}, 65(1):165--241, 2021.

\bibitem{belrhap}
A.~Beligiannis.
\newblock Relative homological algebra and purity in triangulated categories.
\newblock {\em J. Algebra}, 227(1), 2000.

\bibitem{BenGna}
D.~J. Benson and G.~P. Gnacadja.
\newblock Phantom maps and purity in modular representation theory. {I}.
\newblock {\em Fundam. Math.}, 161(1-2), 1999.

\bibitem{bergman}
G.~M. Bergman.
\newblock Every module is an inverse limit of injectives.
\newblock {\em Proc. Amer. Math. Soc.}, 141(4), 2013.

\bibitem{bwdeffun}
I.~Bird and J.~Williamson.
\newblock Definable functors between triangulated categories.
\newblock Preprint, {arXiv}:2310.02159, 2023.

\bibitem{BWhs}
I.~Bird and J.~Williamson.
\newblock The homological spectrum via definable subcategories.
\newblock {\em Bull. Lond. Math. Soc.}, 57(4), 2025.

\bibitem{BvdB}
A.~Bondal and M.~van~den Bergh.
\newblock Generators and representability of functors in commutative and noncommutative geometry.
\newblock {\em Mosc. Math. J.}, 3(1):1--36, 2003.

\bibitem{brown}
E.~H.~j. Brown.
\newblock Cohomology theories.
\newblock {\em Ann. Math. (2)}, 75, 1962.

\bibitem{ckn}
J.~D. Christensen, B.~Keller, and A.~Neeman.
\newblock Failure of {B}rown representability in derived categories.
\newblock {\em Topology}, 40(6), 2001.

\bibitem{CBlfp}
W.~Crawley-Boevey.
\newblock Locally finitely presented additive categories.
\newblock {\em Comm. Algebra}, 22(5), 1994.

\bibitem{gpinj}
G.~Garkusha and M.~Prest.
\newblock Injective objects in triangulated categories.
\newblock {\em J. Algebra Appl.}, 3(4), 2004.

\bibitem{gpzg}
G.~Garkusha and M.~Prest.
\newblock Triangulated categories and the {Z}iegler spectrum.
\newblock {\em Algebr. Represent. Theory}, 8(4), 2005.

\bibitem{gillespiebook}
J.~Gillespie.
\newblock {\em Abelian model category theory}, volume 215 of {\em Cambridge Studies in Advanced Mathematics}.
\newblock Cambridge University Press, Cambridge, 2025.

\bibitem{GJdim}
L.~Gruson and C.~U. Jensen.
\newblock Dimensions cohomologiques reli\'ees aux foncteurs {$\varprojlim\sp{(i)}$}.
\newblock In {\em Paul {D}ubreil and {M}arie-{P}aule {M}alliavin {A}lgebra {S}eminar, 33rd {Y}ear ({P}aris, 1980)}, volume 867 of {\em Lecture Notes in Math.} Springer, Berlin, 1981.

\bibitem{herzog}
I.~Herzog.
\newblock The {Z}iegler spectrum of a locally coherent {G}rothendieck category.
\newblock {\em Proc. London Math. Soc. (3)}, 74(3), 1997.

\bibitem{herzinn}
I.~Herzog and S.~L'Innocente.
\newblock The universal abelian regular ring.
\newblock In {\em Model theory of modules, algebras and categories}, volume 730 of {\em Contemp. Math.} Amer. Math. Soc., [Providence], RI, 2019.

\bibitem{Hovey}
M.~Hovey.
\newblock Classifying subcategories of modules.
\newblock {\em Trans. Amer. Math. Soc.}, 353(8), 2001.

\bibitem{HLP}
M.~Hovey, K.~Lockridge, and G.~Puninski.
\newblock The generating hypothesis in the derived category of a ring.
\newblock {\em Math. Z.}, 256(4), 2007.

\bibitem{axiomatic}
M.~Hovey, J.~H. Palmieri, and N.~P. Strickland.
\newblock Axiomatic stable homotopy theory.
\newblock {\em Mem. Amer. Math. Soc.}, 128(610), 1997.

\bibitem{HZ}
B.~Huisgen-Zimmermann.
\newblock Purity, algebraic compactness, direct sum decompositions, and representation type.
\newblock In {\em Infinite length modules ({B}ielefeld, 1998)}, Trends Math. Birkh\"auser, Basel, 2000.

\bibitem{JL}
C.~U. Jensen and H.~Lenzing.
\newblock {\em Model theoretic algebra: with particular emphasis on fields, rings, modules}, volume~2 of {\em Algebra Log. Appl.}
\newblock New York etc.: Gordon {and} Breach Science Publishers, 1989.

\bibitem{KS}
R.~Kie{\l}pi{\'n}ski and D.~Simson.
\newblock On pure homological dimension.
\newblock {\em Bull. Acad. Polon. Sci. S\'er. Sci. Math. Astronom. Phys.}, 23, 1975.

\bibitem{krausefunlfp}
H.~Krause.
\newblock Functors on locally finitely presented additive categories.
\newblock {\em Colloq. Math.}, 75(1), 1998.

\bibitem{krsmash}
H.~Krause.
\newblock Smashing subcategories and the telescope conjecture---an algebraic approach.
\newblock {\em Invent. Math.}, 139(1), 2000.

\bibitem{krcfsh}
H.~Krause.
\newblock Coherent functors in stable homotopy theory.
\newblock {\em Fund. Math.}, 173(1), 2002.

\bibitem{krcqsl}
H.~Krause.
\newblock Cohomological quotients and smashing localizations.
\newblock {\em Amer. J. Math.}, 127(6), 2005.

\bibitem{lipman}
J.~Lipman.
\newblock Notes on derived functors and {Grothendieck} duality.
\newblock In {\em Foundations of Grothendieck duality for diagrams of schemes}, pages 1--259. Berlin: Springer, 2009.

\bibitem{maclane}
S.~Mac~Lane.
\newblock {\em Categories for the working mathematician.}, volume~5 of {\em Grad. Texts Math.}
\newblock New York, NY: Springer, 2nd ed edition, 1998.

\bibitem{neemantba}
A.~Neeman.
\newblock On a theorem of {B}rown and {A}dams.
\newblock {\em Topology}, 36(3), 1997.

\bibitem{olivier}
J.-P. Olivier.
\newblock Anneaux absolument plats universels et \'epimorphismes d'anneaux.
\newblock {\em C. R. Acad. Sci. Paris S\'er. A-B}, 266, 1968.

\bibitem{psl}
M.~Prest.
\newblock {\em Purity, spectra and localisation}, volume 121 of {\em Encyclopedia of Mathematics and its Applications}.
\newblock Cambridge University Press, Cambridge, 2009.

\bibitem{DAC}
M.~Prest.
\newblock Definable additive categories: purity and model theory.
\newblock {\em Mem. Amer. Math. Soc.}, 210(987), 2011.

\bibitem{SaoStov}
M.~Saor{\'{\i}}n and J.~{\v{S}}t'ov{\'{\i}}{\v{c}}ek.
\newblock {{\(t\)}}-structures with {Grothendieck} hearts via functor categories.
\newblock {\em Sel. Math., New Ser.}, 29(5), 2023.
\newblock Id/No 77.

\bibitem{simsonpgd}
D.~Simson.
\newblock On pure global dimension of locally finitely presented grothendieck categories.
\newblock {\em Fundamenta Mathematicae}, 96(2), 1977.

\bibitem{stack}
T.~{Stacks project authors}.
\newblock The stacks project.
\newblock \url{https://stacks.math.columbia.edu}, 2025.

\bibitem{stenstrom}
B.~Stenstr{\"o}m.
\newblock {\em Rings of quotients. {An} introduction to methods of ring theory}, volume 217 of {\em Grundlehren Math. Wiss.}
\newblock Springer, Cham, 1975.

\bibitem{stovicekpure}
J.~{\v{S}}\v{t}ov\'{\i}\v{c}ek.
\newblock On purity and applications to coderived and singularity categories.
\newblock arXiv.1412.1615, 2014.

\bibitem{weibel}
C.~A. Weibel.
\newblock {\em An introduction to homological algebra}, volume~38 of {\em Cambridge Studies in Advanced Mathematics}.
\newblock Cambridge University Press, Cambridge, 1994.

\end{thebibliography}
\end{document}